\pdfoutput=1
\RequirePackage{ifpdf}
\ifpdf 
\documentclass[pdftex]{sigma}
\else
\documentclass{sigma}
\fi

\usepackage{mathrsfs}

\newtheorem{Theorem}{Theorem}[section]
\newtheorem{Corollary}[Theorem]{Corollary}
\newtheorem{Lemma}[Theorem]{Lemma}
\newtheorem{Claim}[Theorem]{Claim}
\newtheorem{Proposition}[Theorem]{Proposition}
\newtheorem*{Theorem*}{Theorem~\ref{th_criteria}}
\newtheorem*{Theorem**}{Theorem}
 { \theoremstyle{definition}
\newtheorem{Definition}[Theorem]{Definition}
\newtheorem{Example}[Theorem]{Example}
\newtheorem{Remark}[Theorem]{Remark} }

\numberwithin{equation}{section}

\newcommand{\Cc}{\mathscr{C}}
\newcommand{\Lc}{\mathscr{L}}

\begin{document}

\allowdisplaybreaks

\newcommand{\arXivNumber}{1603.07915}

\renewcommand{\PaperNumber}{086}

\FirstPageHeading

\ShortArticleName{Parallelisms \& Lie Connections}

\ArticleName{Parallelisms \& Lie Connections}

\Author{David BL\'AZQUEZ-SANZ~$^\dag$ and Guy CASALE~$^\ddag$}

\AuthorNameForHeading{D.~Bl\'azquez-Sanz and G.~Casale}

\Address{$^\dag$~Universidad Nacional de Colombia, Sede Medell\'in, Facultad de Ciencias, \\
\hphantom{$^\dag$}~Escuela de Matem\'aticas, Calle 59A No. 63 - 20, Medell\'in, Antioquia, Colombia}
\EmailD{\href{mailto:dblazquezs@unal.edu.co}{dblazquezs@unal.edu.co}}

\Address{$^\ddag$~IRMAR, Universit\'e de Rennes 1, Campus de Beaulieu, b\^at.~22-23,\\
\hphantom{$^\ddag$}~263 avenue du G\'en\'eral Leclerc, CS 74205, 35042 RENNES Cedex, France}
\EmailD{\href{mailto:guy.casale@univ-rennes1.fr}{guy.casale@univ-rennes1.fr}}

\ArticleDates{Received September 16, 2016, in f\/inal form October 25, 2017; Published online November 04, 2017}

\Abstract{The aim of this article is to study rational parallelisms of algebraic varieties by means of the transcendence of their symmetries. The nature of this transcendence is measured by a Galois group built from the Picard--Vessiot theory of principal connections.}

\Keywords{parallelism; isogeny; $G$-structure; linear connection; principal connection; dif\/ferential Galois theory}

\Classification{53C05; 14L40; 14E05; 12H05}


\section{Introduction}

The aim of this article is to study rational parallelisms of algebraic varieties by means of the transcendence of their symmetries. Our original motivation was to understand the possible obstructions to the third Lie theorem for algebraic Lie pseudogroups. This article is concerned with the simply transitive case. These obstructions should appear in the Galois group of certain connection associated to a Lie algebroid. However, we have written the article in the language of regular and rational parallelisms of algebraic varieties and their symmetries.

A theorem of P.~Deligne says that any Lie algebra can be realized as a~parallelism of an algebraic variety. This is a sort of algebraic version of the third Lie theorem. Notwithstanding, there is one main problem: given an algebraic variety with a parallelism, how far is it from being an algebraic group? Is it possible to conjugate this parallelism with the canonical parallelism of invariant vector f\/ields on an algebraic group?

In the analytic context, from the Darboux--Cartan theorem \cite[p.~212]{sharpe}, a $\mathfrak{g}$-parallelized complex manifold $M$ has a natural $(G,G)$ structure where $G$ is a Lie group with $\mathfrak{lie}(G) = \mathfrak{g}$. The obstruction to be a covering of~$G$, as manifold with a~$(G,G)$ structure, is contained in a mo\-no\-dromy group \cite[p.~130]{sharpe}. In~\cite{Wang}, Wang proved that parallelized compact complex manifolds are biholomorphic to quotients of complex Lie groups by discrete cocompact subgroups. This result has been extended by Winkelmann in \cite{Winkelmann1, Winkelmann2} for some open complex manifolds.

In this article we address the problem of classif\/ication of rational parallelisms on algebraic varieties up to birational transformations. Such a classif\/ication seems impossible in the algebraic category but we prove a criterion to ensure that a parallelized algebraic variety is isogenous to an algebraic group. Summarizing, we pursue the following plan: We regard inf\/initesimal symmetries of a rational parallelism as horizontal sections of a connection that we call the reciprocal Lie connection. This connection has a Galois group which is represented as a group of internal automorphisms of a Lie algebra. The obstruction to the algebraic conjugation to an algebraic group, under some assumptions, appear in the Lie algebra of this Galois group.

In Section \ref{section_parallelisms} we introduce the basic def\/initions; several examples of parallelisms are given here. In Section \ref{section_lie} we study the properties of connections on the tangent bundle whose local analytic horizontal sections form a sheaf of Lie algebras of vector f\/ields. We call them {\it Lie connections}. They always come by pairs, and they are characterized by having vanishing curvature and constant torsion (Proposition~\ref{prop_Lie_char}). We see that each rational parallelism has an attached pair of Lie connections, one of them with trivial Galois group. We compute the Galois groups of some parallelisms given in examples (Proposition \ref{prop_example}), and prove that any algebraic subgroup of ${\rm PSL}_2(\mathbf C)$ appears as the dif\/ferential Galois group of a $\mathfrak{sl}_2(\mathbf C)$-parallelism (Theorem~\ref{thm_SL2}). Section~\ref{section_DC} is devoted to the construction of the isogeny between a $\mathfrak g$-parallelized variety and an algebraic group $G$ whose Lie algebra is~$\mathfrak g$. In order to do this, we introduce the Darboux--Cartan connection, a $G$-connection whose horizontal sections are parallelism conjugations. We prove that if $\mathfrak g$ is centerless then the Darboux--Cartan connection and the reciprocal Lie connection have isogenous Galois groups. We prove that the only centerless Lie algebras $\mathfrak{g}$ such that there exists a $\mathfrak{g}$-parallelism with a trivial Galois group are algebraic Lie algebras, i.e., Lie algebras of algebraic groups. In particular this allows us to give a criterion for a parallelized variety to be isogenous to an algebraic group. The vanishing of the Lie algebra of the Galois group of the reciprocal connection is a necessary and suf\/f\/icient condition for a parallelized variety to be isogenous to an algebraic group:

\begin{Theorem*}
Let $\mathfrak g$ be a centerless Lie algebra. An algebraic variety $(M,\omega)$ with a rational parallelism of type $\mathfrak g$ is isogenous to an algebraic group if and only if $\mathfrak{gal}(\nabla^{\rm rec}) = \{0\}$.
\end{Theorem*}

The notion of {\it isogeny} can be extended beyond the simply-transitive case. Let us consider a~complex Lie algebra $\mathfrak g$. An {\it infinitesimally homogeneous variety} of type $\mathfrak g$ is a pair $(M,\mathfrak s)$ consisting of a complex smooth irreducible variety $M$ and a f\/inite-dimensional Lie algebra \smash{$\mathfrak s \subset \mathfrak X(M)$} isomorphic to $\mathfrak g$ that spans the tangent bundle of $M$ on the generic point.

We are interested in conjugation by rational or by algebraic maps, so that, whenever necessary, we replace $M$ by a suitable Zariski open subset. In this context, we say that a dominant rational map $f\colon M_1 \dasharrow M_2$ between varieties of the same dimension conjugates the inf\/initesimally homogeneous varieties $(M_1,\mathfrak s_1)$ and $(M_2,\mathfrak s_2)$ if $f^*(\mathfrak s_2) = \mathfrak s_1$. We say that $(M_1,\mathfrak s_1)$ and $(M_2,\mathfrak s_2)$ are {\it isogenous} if they are conjugated to the same inf\/initesimally homogeneous space of type~$\mathfrak g$.

Under some hypothesis on the Lie algebra $\mathfrak s\subset \mathfrak X(M)$ one can prove that $(M,\mathfrak s)$ is isogenous to a homogeneous space $(G/H,\mathfrak{lie}(G)^{\rm rec})$ with the action of right invariant vector f\/ields. These hypothesis are satisf\/ied by transitive actions of $\mathfrak{sl}_{n+1}(\mathbf C)$ on $n$-dimensional varieties. As a particular case of Theorem~\ref{homogeneous} one has
\begin{Theorem**}
Let $(M,\mathfrak s)$ be an infinitesimally homogeneous variety of complex dimension $n$ such that $\mathfrak s$ is isomorphic to $\mathfrak{sl}_{n+1}(\mathbf C)$. Then there exists a dominant rational map $M \dasharrow \mathbf{CP}_n$ conjugating $\mathfrak s$ with the Lie algebra $\mathfrak{sl}_{n+1}(\mathbf C)$ of projective vector fields in $\mathbf{CP}_n$.
\end{Theorem**}

Appendix~\ref{ApA} is devoted to a geometrical presentation of Picard--Vessiot theory for linear and principal connections. Finally, Appendix~\ref{apB} contains a detailed proof of Deligne's theorem of the realization of a regular parallelism modeled over any f\/inite-dimensional Lie algebra. This includes also a computation of the Galois group that turns out to be, for this particular construction, an algebraic torus.

\section{Parallelisms}\label{section_parallelisms}

Let $M$ be a smooth connected af\/f\/ine variety over $\mathbf C$ of dimension~$r$. We denote by $\mathbf C[M]$ its ring of regular functions and by $\mathbf C(M)$ its f\/ield of rational functions. Analogously, we denote by~$\mathfrak X[M]$ and~$\mathfrak X(M)$ respectively the Lie algebras of regular and rational vector f\/ields in $M$, and so on.

Let $\mathfrak g$ be a Lie algebra of dimension $r$. We f\/ix a basis $A_1,\ldots,A_r$ of $\mathfrak g$, and the following notation for the associated structure constants $[A_i,A_j] = \sum_{k}\lambda_{ij}^kA_k$.

A parallelism of type $\mathfrak g$ of $M$ is a realization of the Lie algebra $\mathfrak g$ as a Lie algebra of pointwise linearly independent vector f\/ields in $M$. More precisely:

\begin{Definition}A regular parallelism of type $\mathfrak g$ in $M$ is a Lie algebra morphism, $\rho\colon \mathfrak g \to \mathfrak X[M]$ such that $\rho A_1(x), \ldots, \rho A_r(x)$ form a basis of $T_xM$ for any point $x$ of $M$.
\end{Definition}

\begin{Example}\label{ex:AGP}Let $G$ be an algebraic group and $\mathfrak g$ be its Lie algebra of left invariant vector f\/ields. Then the natural inclusion $\mathfrak g\subset \mathfrak X[G]$ is a regular parallelism of~$G$. The Lie algebra $\mathfrak g^{\rm rec}$ of right invariant vector f\/ields is another regular parallelism of the same type. Let invariant and right invariant vector f\/ields commute, hence, an algebraic group is naturally endowed with a~pair of commuting parallelisms of the same type.
\end{Example}

From Example \ref{ex:AGP}, it is clear that any \emph{algebraic} Lie algebra is realized as a parallelism of some algebraic variety. On the other hand, Theorem~\ref{TDeligne} due to P.~Deligne and published in~\cite{Malgrange}, ensures that any Lie algebra is realized as a regular parallelism of an algebraic variety. Analogously, we have the def\/initions of rational and local analytic parallelism. Note that a~rational parallelism in~$M$ is a regular parallelism in a Zariski open subset $M^\star \subseteq M$.

There is dual def\/inition, equivalent to that of parallelism. This is more suitable for calculations.

\begin{Definition}A regular parallelism form (or coparallelism) of type $\mathfrak g$ in $M$ is a $\mathfrak g$-valued
$1$-form $\omega\in\Omega^1[M]\otimes_{\mathbb C} \mathfrak g$ such that:
\begin{itemize}\itemsep=0pt
\item[(1)]For any $x\in M$, $\omega_x\colon T_x M \to \mathfrak g$ is a linear isomorphism.
\item[(2)]If $A$ and $B$ are in $\mathfrak g$ and $X$, $Y$ are vector f\/ields such that $\omega(X) = A$ and $\omega (Y) = B$ then $\omega[X,Y] = [A,B]$.
\end{itemize}
\end{Definition}

Analogously, we def\/ine local analytic and rational coparallelism of type $\mathfrak g$ in $M$. It is clear that each coparallelism induces a parallelism, and reciprocally, by the relation $\omega (\rho (A)) = A$. Thus, there is a natural equivalence between the notions of parallelism and coparallelism. From now on we f\/ix $\rho$ and $\omega$ equivalent parallelism and coparallelism of type $\mathfrak g$ on $M$.

The Lie algebra structure of $\mathfrak g$ forces $\omega$ to satisfy Maurer--Cartan structure equations
\begin{gather*}{\rm d}\omega + \frac{1}{2}[\omega,\omega] = 0.\end{gather*}
Taking components $\omega = \sum_{i}\omega_iA_i$ we have
\begin{gather*}{\rm d}\omega_i +\sum_{j,k=1}^r \frac{1}{2}\lambda_{jk}^i\omega_j\wedge \omega_k = 0.\end{gather*}

\begin{Example}\label{ex:group} Let $G$ be an algebraic group and $\mathfrak g$ be the Lie algebra of left invariant vector f\/ields in $G$. Then the structure form $\omega$ is the coparallelism corresponding to the parallelism of Example~\ref{ex:AGP}.
\end{Example}

\begin{Example}\label{ex:B}Let $\mathfrak g = \langle A_1,A_2\rangle$ be the 2-dimensional Lie algebra with commutation relation
\begin{gather*}[A_1, A_2]= A_1.\end{gather*}
The vector f\/ields
\begin{gather*}X_1 = \frac{\partial}{\partial x},\qquad X_2 = x\frac{\partial}{\partial x}+ \frac{\partial}{\partial y},\end{gather*}
def\/ine a regular parallelism via $\rho (A_i)= X_i$ of $\mathbf C^2$. The associated parallelism form is
\begin{gather*}\omega = A_1{\rm d}x + (A_2 - x A_1){\rm d}y.\end{gather*}
\end{Example}

\begin{Example}[Malgrange]\label{ex:MD} Let $\mathfrak g = \langle A_1,A_2,A_3 \rangle$ be the 3-dimensional Lie algebra with commutation relations
\begin{gather*}[A_1, A_2]= \alpha A_2,\qquad [A_1,A_3]=\beta A_3, \qquad [A_2,A_3] = 0,\end{gather*}
with $\alpha$, $\beta$, non zero complex numbers. In particular, if $\alpha/\beta$ is not rational then $\mathfrak g$ is not the Lie algebra of an algebraic group. The vector f\/ields
\begin{gather*}X_1 = \frac{\partial}{\partial x} + \alpha y \frac{\partial}{\partial y} + \beta z\frac{\partial}{\partial z},\qquad
X_2 = \frac{\partial}{\partial y},\qquad X_3 = \frac{\partial}{\partial z},\end{gather*}
def\/ine a regular parallelism via $\rho (A_i)= X_i$ of $\mathbf C^3$. The associated parallelism form is
\begin{gather*}\omega = (A_1 - A_2\alpha y - A_3\beta z){\rm d}x + A_2{\rm d}y + A_3{\rm d}z.\end{gather*}
\end{Example}

\begin{Definition} \label{isogenous} Let $(M,\omega)$ and $(N,\theta)$ be algebraic manifolds with coparallelisms of type~$\mathfrak g$. We say that they are isogenous if there is an algebraic manifold $(P,\eta)$ with a coparallelism of type~$\mathfrak g$ and dominant maps $f\colon P\to M$ and $g\colon P \to N$ such that $f^*(\omega) = g^*(\theta) = \eta$.
\end{Definition}

Clearly, the notion of isogeny of parallelized varieties extends that of isogeny of algebraic groups.

\begin{Example}\label{ex:cover} Let $f\colon M\dasharrow G$ be a dominant rational map with values in an algebraic group with $\dim_{\mathbf{C}}M =\dim_{\mathbf{C}}G$. Then $\theta =f^*(\omega)$ is a rational parallelism form in $M$.
\end{Example}

\begin{Example}\label{ex:finite}Let $H$ be a f\/inite subgroup of the algebraic group $G$ and
\begin{gather*}\pi\colon \ G\to M = H\setminus G = \{Hg \colon g\in G\}\end{gather*} be the quotient by the action of $H$ on the left side. The structure form $\omega$ in $G$ is left-invariant and then it is projectable by $\pi$. Then, $\theta = \pi_*(\omega)$ is a regular parallelism form in~$M$.
\end{Example}

\begin{Example}\label{ex:coverfinite} Combining Examples \ref{ex:cover} and \ref{ex:finite}, let $H\subset G$ be a f\/inite subgroup and $f\colon M\to H \setminus G$ be a dominant rational map between manifolds of the same dimension. Then $\theta = f^*(\pi_*(\omega))$ is a rational parallelism form in $M$.
\end{Example}

\begin{Example}By application of Example \ref{ex:coverfinite} to the case of the multiplicative group we obtain rational multiples of logarithmic forms in $\mathbf{CP}_1$, $\frac{p}{q}\frac{{\rm d}f}{f}$ where $f\in\mathbf C(z)$. Thus, rational multiples of logarithmic forms in $\mathbf{CP}_1$ are the rational coparallelisms isogenous to that of the multiplicative group.
\end{Example}

\begin{Example}By application of Example \ref{ex:coverfinite} to the case of the additive group we obtain the exact forms in $\mathbf{CP}_1$, ${\rm d}F$ where $F\in\mathbf C(z)$. Thus, the exact forms in $\mathbf{CP}_1$ are the rational coparallelisms isogenous to that of the additive group.
\end{Example}

\begin{Example}\label{ex:quotient_c}Let $H$ be a subgroup of the algebraic group $G$, with Lie algebra $\mathfrak h\subset \mathfrak g$. Let us
assume that $\mathfrak h$ admits a supplementary Lie algebra $\mathfrak h'$
\begin{gather*}\mathfrak g = \mathfrak h \oplus \mathfrak h' \qquad \mbox{(as vector spaces).}\end{gather*}
We consider the left quotient $M= H \setminus G$ of $G$ by the action of $H$ and the quotient map $\pi\colon G\to M$. It turns out that $\mathfrak h'$ is a Lie algebra of vector f\/ields in $G$ projectable by $\pi$, and thus $\pi_*|_{\mathfrak h'} \colon \mathfrak h' \to \mathfrak X[M]$ gives a parallelism of $M$ that is regular in the open subset
\begin{gather*}\{Hg\in M \colon \operatorname{Adj}_g(\mathfrak h) \cap \mathfrak h' = \{0\} \}.\end{gather*}
It turns out to be regular in $M$ if $H\lhd G$. Examples~\ref{ex:B} and~\ref{ex:MD} are particular cases where $G$ is $\operatorname{Af\/f}(2,\mathbf C)$ and $\operatorname{Af\/f}(3,\mathbf C)$ respectively.
\end{Example}

\begin{Remark}We can see also Example~\ref{ex:quotient_c} as a coparallelism. Let $\pi'\colon\mathfrak g\to \mathfrak h'$ be the projection given by the vector space decomposition $\mathfrak g = \mathfrak h \oplus \mathfrak h'$. Since $\pi'\circ \omega$ is left invariant form in $G$, it is projectable by~$\pi$. Hence, there is a form $\omega'$ in $M$ such that $\pi^*\omega' = \pi'\circ \omega$. This form $\omega'$ is the corresponding coparallelism.
\end{Remark}

\section{Associated Lie connection}\label{section_lie}

\subsection{Reciprocal connections}

Let $\nabla$ be a linear connection (rational or regular) on $TM$. The reciprocal connection is def\/ined as
\begin{gather*}\nabla^{\rm rec}_{\vec X}\vec Y = \nabla_{\vec Y} \vec X + \big[\vec X,\vec Y\big].\end{gather*}
From this def\/inition it is clear that the dif\/ference $\nabla - \nabla^{\rm rec} = \operatorname{Tor}_{\nabla}$ is the torsion tensor, $\operatorname{Tor}_{\nabla} = -\operatorname{Tor}_{\nabla^{\rm rec}}$ and $(\nabla^{\rm rec})^{\rm rec} = \nabla$.

\subsection{Connections and parallelisms}

Let $\omega$ be a coparallelism of type $\mathfrak g$ in $M$ and $\rho$ its equivalent parallelism. Denote by $\vec X_i$ the basis of vector f\/ields in $M$ such that $\omega(\vec X_i)=A_i$ is a basis of $\mathfrak g$.

\begin{Definition}The connection $\nabla$ associated to the parallelism $\omega$ is the only linear connection in $M$ for which $\omega$ is a $\nabla$-horizontal form.
\end{Definition}

Clearly $\nabla$ is a f\/lat connection and the basis $\{\vec X_i\}$ is a basis of the space of $\nabla$-horizontal vector f\/ields. In this basis $\nabla$ has
vanishing Christof\/fel symbols
\begin{gather*}\nabla_{\vec X_i} \vec X_j = 0.\end{gather*}

Let us compute some inf\/initesimal symmetries of $\omega$. A vector f\/ield $\vec Y$ is an inf\/initesimal symmetry of $\omega$ if ${\rm Lie}_{\vec Y}\omega = 0$, or equivalently, if it commutes with all the vector f\/ields of the parallelism
\begin{gather*}\big[\vec X_i, \vec Y\big] = 0, \qquad i = 1,\ldots, r.\end{gather*}

\begin{Lemma}\label{Lemma1}Let $\nabla$ be the connection associated to the parallelism $\omega$. Then for any vector field~$\vec Y$ and any $j=1,\ldots,r$
\begin{gather*}\big[\vec X_j, \vec Y\big] = \nabla^{\rm rec}_{\vec X_j}{\vec Y}.\end{gather*}
Thus, $\vec Y$ is an infinitesimal symmetry of $\omega$ if and only if it is a~horizontal vector field for the reciprocal connection $\nabla^{\rm rec}$.
\end{Lemma}

\begin{proof}A direct computation yields the result. Take $\vec Y = \sum\limits_{k=1}^r f_k\vec X_k$, for each $j$ we have
\begin{gather*}\nabla^{\rm rec}_{\vec X_j} \vec Y = \sum_{k=1}^r\big( \big(\vec X_jf_k\big)\vec X_k + f_k\big[\vec X_j, \vec X_k\big]\big) = \big[\vec X_j,\vec Y\big]. \tag*{\qed}\end{gather*}\renewcommand{\qed}{}
\end{proof}

The above considerations also give us the Christof\/fel symbols for $\nabla^{\rm rec}$ in the basis $\{\vec X_i\}$
\begin{gather*}\nabla^{\rm rec}_{\vec X_i}\vec X_j = \big[\vec X_i, \vec X_j\big] = \sum_{k=1}^r \lambda_{ij}^k \vec X_k,\end{gather*}
i.e., the Christof\/fel symbols of $\nabla^{\rm rec}$ are the structure constants of the Lie algebra $\mathfrak g$.

\begin{Lemma}\label{Lemma2} Let $\nabla$ be the connection associated to a coparallelism in $M$. Then, $\nabla^{\rm rec}$ is flat, and the Lie bracket of two $\nabla^{\rm rec}$-horizontal vector fields is a $\nabla^{\rm rec}$-horizontal vector field.
\end{Lemma}

\begin{proof}The f\/latness and the preservation of the Lie bracket by $\nabla^{\rm rec}$ are direct consequences of the Jacobi identity. Let us compute the curvature
\begin{gather*}R\big(\vec X_i,\vec X_j,\vec X_k\big) = \nabla^{\rm rec}_{\vec X_i}\big(\nabla^{\rm rec}_{\vec X_j} X_k\big)
- \nabla^{\rm rec}_{\vec X_j}\big(\nabla^{\rm rec}_{\vec X_i} \vec X_k\big)
- \nabla^{\rm rec}_{[\vec X_i,\vec X_j]}\vec X_k\\
\hphantom{R\big(\vec X_i,\vec X_j,\vec X_k\big)}{} = \rho ([A_i,[A_j,A_k]] - [A_j,[A_i, A_k]] - [[A_i,A_j],A_k]) = 0.\end{gather*}
Let us compute the Lie bracket for $\vec Y$ and $\vec Z$ $\nabla^{\rm rec}$-horizontal vector f\/ields
\begin{gather*}\nabla^{\rm rec}_{\vec X_i}\big[\vec Y, \vec Z\big]\! = \big[\vec X_i,\big[\vec Y, \vec Z\big]\big]\! =
\big[\big[\vec X_i, \vec Y\big], \vec Z\big]\! + \big[ \vec Y, \big[\vec X_i, \vec Z \big]\big]\! =
\big[\nabla^{\rm rec}_{\vec X_i}\vec Y, \vec Z\big]\! +\big[\vec Y, \nabla^{\rm rec}_{\vec X_i}\vec Z\big]\! = 0.\!\!\!\!\!\!\tag*{\qed}\end{gather*}\renewcommand{\qed}{}
\end{proof}

\begin{Lemma}\label{Lemma3} Let $x\in M$ be a regular point of the parallelism form $\omega$. The space of germs at~$x$ of horizontal vector fields for $\nabla^{\rm rec}$ is a Lie algebra isomorphic to $\mathfrak g$. Moreover, let $\vec Y_1,\ldots, \vec Y_r$ be horizontal vector fields with initial conditions $\vec Y_i(x) = \vec X_i(x)$, then $\big[\vec Y_i, \vec Y_j\big] = - \sum\limits_{k=1}^r \lambda_{ij}^k \vec Y_k$, where the $\lambda_{i,j}$ are the structure constants of the Lie algebra generated by the $\vec X_i$.
\end{Lemma}

\begin{proof} We can write the vector f\/ields $\vec Y_i$ as linear combinations of the vector f\/ields $\vec X_i$: $\vec Y_i = \sum\limits_{j=1}^r a_{ji}\vec X_j$. The matrix $(a_{ij})$ satisf\/ies the dif\/ferential equation
\begin{gather*}\vec X_{k}a_{ij} = - \sum_{\alpha = 1}^r \lambda_{k \alpha}^i a_{\alpha j}, \qquad a_{ij}(x) = \delta_{ij}.\end{gather*}
On the other hand, we have $\big[\vec Y_i, \vec Y_j\big](x) = \sum\limits_{k=1}^r \hat\lambda_{ij}^k \vec Y_k (x)$, for certain unknown structure cons\-tants~$\hat\lambda_{ij}^k$. Let us check that $\hat\lambda_{ij}^k = \lambda_{ji}^k = -\lambda_{ij}^k$,
\begin{gather*}\big[\vec Y_i,\vec Y_j\big] = \left[\sum_{\alpha=1}^r a_{\alpha i} \vec X_\alpha, \sum_{\beta = 1}^r a_{\beta j} \vec X_{\beta} \right] \\
\hphantom{\big[\vec Y_i,\vec Y_j\big]}{} = \sum_{\alpha, \beta, \gamma =1}^r -a_{\alpha i}\lambda_{\alpha \gamma}^\beta a_{\gamma j}\vec X_\beta
+ \sum_{\alpha, \beta, \gamma =1}^r a_{\beta j}\lambda_{\beta\gamma}^\alpha a_{\gamma i} \vec X_{\alpha}
+ \sum_{\alpha, \beta, \gamma =1}^r a_{\beta j} a_{\alpha i} \lambda_{\alpha \beta}^\gamma \vec X_\gamma.\end{gather*}
Taking values at $x$, we obtain
\begin{gather*}\big[\vec Y_i, \vec Y_j\big](x) = \sum_{\beta =1}^r - \lambda_{i j}^\beta \vec Y_\beta(x)
+ \sum_{\alpha =1}^r \lambda_{j i}^\alpha \vec Y_{\alpha}(x) + \sum_{\gamma =1}^r \lambda_{i j}^\gamma \vec Y_\gamma(x) =
\sum_{\alpha =1}^r \lambda_{j i}^k \vec Y_{k}(x).\tag*{\qed}\end{gather*}\renewcommand{\qed}{}
\end{proof}

\begin{Example}Let $G$ be an algebraic group with Lie algebra $\mathfrak g$. As seen in Example~\ref{ex:group} the Maurer--Cartan structure form $\omega$ is a coparallelism in~$G$. Let $\nabla$ be the connection associated to this coparallelism. There is another canonical coparallelism, the right invariant Maurer--Cartan structure form $\omega_{\rm rec}$, let us consider ${\bf i}\colon G\to G$ the inversion map,
\begin{gather*}\omega_{\rm rec} = - {\bf i}^*(\omega).\end{gather*}
As may be expected, the connection associated to the coparallelism $\omega_{\rm rec}$ is $\nabla^{\rm rec}$. Right invariant vector f\/ields in $G$ are inf\/initesimal symmetries of left invariant vector f\/ields and vice versa. In this case, the horizontal vector f\/ields of $\nabla$ and $\nabla^{\rm rec}$ are regular vector f\/ields.
\end{Example}

As shown in the next three examples, symmetries of a rational parallelism are not in general rational vector f\/ields.

\begin{Example}Let us consider the Lie algebra $\mathfrak g$ and the coparallelism $\omega = A_1{\rm d}x + (A_2 - x A_1){\rm d}y$, of Example~\ref{ex:B}. Let $\nabla$ be its associated connection. In cartesian coordinates, the only non-vanishing Christof\/fel symbol of the reciprocal connection is $\Gamma_{21}^1 = -1$. A basis of $\nabla^{\rm rec}$-horizontal vector f\/ields is
\begin{gather*}\vec Y_1 = e^y\frac{\partial}{\partial x}, \qquad \vec Y_2 = \frac{\partial}{\partial y}.\end{gather*}
Note that they coincide with $\vec X_1$, $\vec X_2$ at the origin point and $\big[\vec Y_1,\vec Y_2\big] = - Y_1$.
\end{Example}

\begin{Example} Let us consider the Lie algebra $\mathfrak g$ and the coparallelism $\omega = (A_1 - \alpha y A_2 - \beta z A_3){\rm d}x + A_2{\rm d}y + A_3{\rm d}z$ of Example~\ref{ex:MD}. Let $\nabla$ be its associated connection. In cartesian coordinates, the only non-vanishing Christof\/fel symbols of the reciprocal connection are
\begin{gather*}\Gamma_{11}^2 = -\alpha,\qquad \Gamma_{11}^3 = -\beta.\end{gather*}
A basis of $\nabla^{\rm rec}$-horizontal vector f\/ields is
\begin{gather*}\vec Y_1 = \frac{\partial}{\partial x}, \qquad \vec Y_2 = e^{\alpha x}\frac{\partial}{\partial y},\qquad \vec Y_3 = e^{\beta x}\frac{\partial}{\partial z}.\end{gather*}
Note that they coincide with $\vec X_1$, $\vec X_2$, $\vec X_3$ at the origin point and
\begin{gather*}\big[\vec Y_1,\vec Y_2\big] = - \alpha Y_2,\qquad \big[\vec Y_1,\vec Y_3\big] = -\beta \vec Y_3.\end{gather*}
\end{Example}

\begin{Example}Let us consider the Lie algebra $\mathfrak g$ of Example~\ref{ex:MD} and the coparallelism
\begin{gather*}\omega = (A_1 - \alpha y A_2 - \beta z A_3)\frac{{\rm d}x}{x} + A_2{\rm d}y + A_3{\rm d}z.\end{gather*}
Let $\nabla$ be its associated connection. In cartesian coordinates, the only non-vanishing Christof\/fel symbols of the reciprocal connection are
\begin{gather*}\Gamma_{11}^2 = -\alpha,\qquad \Gamma_{11}^3 = -\beta.\end{gather*}
A basis of $\nabla^{\rm rec}$-horizontal vector f\/ields on a simply connected open subspace $U\subset \mathbb C^\ast \times \mathbb C^2$ is
\begin{gather*}\vec Y_1 = x\frac{\partial}{\partial x}, \qquad \vec Y_2 = x^{\alpha}\frac{\partial}{\partial y}, \qquad \vec Y_3 = x^{\beta}\frac{\partial}{\partial z}.\end{gather*}
\end{Example}

\subsection{Lie connections}

The connections $\nabla$ and $\nabla^{\rm rec}$ associated to a coparallelism $\omega$ of type $\mathfrak g$ are particular cases of the following def\/inition.

\begin{Definition} A Lie connection (regular or rational) in $M$ is a f\/lat connection $\nabla$ in $TM$ such that the Lie bracket of any two horizontal vector f\/ields is a horizontal vector f\/ield.
\end{Definition}

Given a Lie connection $\nabla$ in $M$, there is a $r$-dimensional Lie algebra $\mathfrak g$ such that the space of germs of horizontal vector f\/ields at a regular point $x$ is a~Lie algebra isomorphic to $\mathfrak g$. We will say that $\nabla$ is a Lie connection of type~$\mathfrak g$. The following result gives several algebraic characterizations of Lie connections:

\begin{Proposition}\label{prop_Lie_char} Let $\nabla$ be a linear connection in $TM$, the following statements are equivalent:
\begin{itemize}\itemsep=0pt
\item[$(1)$] $\nabla$ is a Lie connection;
\item[$(2)$] $\nabla^{\rm rec}$ is a Lie connection;
\item[$(3)$] $\nabla$ is flat and has constant torsion, $\nabla \operatorname{Tor}_{\nabla} = 0$;
\item[$(4)$] $\nabla$ and $\nabla^{\rm rec}$ are flat.
\end{itemize}
\end{Proposition}

\begin{proof} Let us f\/irst see (1)$\Leftrightarrow$(2). Let $\nabla$ be a Lie connection. Around each point of the domain of $\nabla$ there is a parallelism, by possibly transcendental vector f\/ields, such that $\nabla$ is its associated connection. Then, Lemma~\ref{Lemma2} states (1)$\Rightarrow$(2). Taking into account that $(\nabla^{\rm rec})^{\rm rec} = \nabla$ we have the desired equivalence.

Let us see now that (1)$\Leftrightarrow$(3). Let us assume that $\nabla$ is a f\/lat connection. For any three vector f\/ields $X$, $Y$, $Z$ in $M$ we have
\begin{gather*}(\nabla_X \operatorname{Tor}_{\nabla})(Y,Z) = - \operatorname{Tor}_{\nabla}(\nabla_X Y,Z ) - \operatorname{Tor}_{\nabla}(Y,\nabla_X Z) + \nabla_X \operatorname{Tor}_{\nabla}(Y,Z).
\end{gather*}
Let us assume that $Y$ and $Z$ are $\nabla$-horizontal vector f\/ields. Then, we have
\begin{gather*} \operatorname{Tor}_{\nabla}(Y,Z) = \nabla_Y Z - \nabla_Z Y - [Y,Z] = - [Y,Z]\end{gather*}
and the previous equality yields
\begin{gather*}(\nabla_X \operatorname{Tor}_{\nabla})(Y,Z) = - \nabla_X[Y,Z].\end{gather*}
Thus, we have that $\nabla \operatorname{Tor}_{\nabla}$ vanishes if and only if the Lie bracket of any two $\nabla$-horizontal vector f\/ields is also $\nabla$-horizontal. This proves (1)$\Leftrightarrow$(3).

Finally, let us see (1)$\Leftrightarrow$(4). It is clear that (1) implies (4) so we only need to see (4)$\Rightarrow$(1). Assume $\nabla$ and $\nabla^{\rm rec}$ are f\/lat. Then, locally, there exist a basis $\{\vec X_i\}$ of $\nabla$-horizontal vector f\/ields and a~basis $\{\vec Y_i\}$ of $\nabla^{\rm rec}$-horizontal vector f\/ields. By the def\/inition of the reciprocal connection, we have that a vector f\/ield $\vec X$ is $\nabla$-horizontal if and only if it satisf\/ies $[\vec X, \vec Y_i]=0$ for $i=1,\ldots,r$. By the Jacobi identity we have
\begin{gather*}\big[\big[\vec X_i,\vec X_j\big],\vec Y_k\big]= 0.\end{gather*}
The Lie brackets $\big[\vec X_i,\vec X_j\big]$ are also $\nabla$-horizontal and $\nabla$ is a Lie connection.
\end{proof}

\begin{Lemma} Let $\nabla$ be a Lie connection on $M$. Let $x$ be a regular point and $\vec X_1,\ldots, \vec X_r$ and $\vec Y_1,\ldots, \vec Y_r$ be basis of horizontal vector field germs on $M$ for $\nabla$ and~$\nabla^{\rm rec}$ respectively with same initial conditions $\vec X_i(x) = \vec Y_i(x)$. Then
\begin{gather*}\big[\vec X_i,\vec X_j\big](x) = - \big[\vec Y_i,\vec Y_j\big](x).\end{gather*}
It follows that $\nabla$ and $\nabla^{\rm rec}$ are of the same type $\mathfrak g$.
\end{Lemma}

\begin{proof} By def\/inition $\nabla$ is the connection associated to the local analytic parallelism given by the basis $\{\vec X_i\}$ of horizontal vector f\/ields. Then we apply Lemma~\ref{Lemma3} in order to obtain the desired conclusion.
\end{proof}

\subsection{Some results on Lie connections by means of Picard--Vessiot theory}

Def\/initions and general results concerning the Picard--Vessiot theory of connections are given in Appendix~\ref{ApA}.

\begin{Proposition} Let $\nabla$ be a rational Lie connection in $TM$. The $\nabla$-horizontal vector fields are the symmetries of a rational parallelism of $M$ if and only if $\operatorname{Gal}(\nabla^{\rm rec}) = \{1\}$.
\end{Proposition}

\begin{proof} We will use the notations of Section~\ref{A6}: $R^1(TM)$ is the ${\rm GL}_n(\mathbb{C})$-principal bundle associated to $TM$ and $\mathcal F'$ is the ${\rm GL}_n(\mathbb{C})$-invariant foliation on $R^1(TM)$ given by graphs of local basis of $\nabla$-horizontal sections. The Galois group $\operatorname{Gal}(\nabla^{\rm rec})$ can be computed as soon as we know the Zariski closure $\overline{\Lc}$ of a leaf $\Lc$ of the induced foliation $\mathcal F'$ on $R^1(TM)$. $\operatorname{Gal}(\nabla^{\rm rec})$ is f\/inite is and only if $\overline{\Lc} = \Lc$ and is $\{1\}$ if and only if $\Lc$ is the graph of a rational section $M \to R^1(TM)$. This means that there exists a basis of rational $\nabla^{\rm rec}$-horizontal sections. These sections give the desired parallelism.
\end{proof}

\begin{Proposition} For any Lie connection $\nabla$, $\operatorname{Gal}(\nabla)\subseteq \operatorname{Aut}(\mathfrak g)$.
\end{Proposition}

\begin{proof}Let us choose a point $x \in M$ regular for $\nabla$ and a basis $A_1,\ldots, A_r$ of $\mathfrak{g}$, i.e., a basis $Y_1,\ldots, Y_r$ of local $\nabla$-horizontal section of~$TM$ at~$x$.

Using notation of Section~\ref{A6}, we will identify $R^1(T_xM)$ with the set of isomorphisms of linear spaces $ \sigma\colon \mathfrak g \to T_xM$; now $\operatorname{Gal}(\nabla)\subseteq {\rm GL}(\mathfrak g)$. Because of the construction of $\mathfrak g$, we have a~canonical point in~$R^1(TM)$ corresponding to the identity $ \sigma_o\colon \mathfrak g \to T_xM$.

For $m \in M$, if $\sigma$ is an isomorphism from $\mathfrak g$ to $T_{m}M$ then one def\/ines $H^k_{{i,j}}(\sigma)$ to be
\begin{gather*}\frac{[X_{i},X_{j}] \wedge X_{1}\wedge \cdots \wedge \widehat{X_{k}}\wedge \cdots \wedge X_{r}}{X_{k} \wedge X_{1}\wedge \cdots \wedge \widehat{X_{k}} \wedge \cdots \wedge X_{r}} \Big|_m, \end{gather*} where $X_i$ is the horizontal section such that $X_i(m)= \sigma A_i$. These functions are regular functions on~$R^1(TM)$. Moreover they are constant and equal to the constant structures on the Zariski closure of the leaf passing through~$\sigma_o$. The Galois group is the stabilizer of this leaf then the functions~$H^k_{{i,j}}$ are invariant under the action of the Galois group, i.e., the Galois group preserves the Lie bracket.
\end{proof}

\begin{Proposition}\label{prop_example} Let $\mathfrak h'$ be a Lie sub-algebra of the Lie algebra of some algebraic group and let $G$ be the smallest algebraic subgroup such that ${\mathfrak{lie}}(G) = \mathfrak g \supset \mathfrak h'$. Assume the existence of an algebraic subgroup $H$ of $G$ whose Lie algebra $\mathfrak h$ is supplementary to $\mathfrak h'$ in $\mathfrak g$, $\mathfrak g = \mathfrak h \oplus \mathfrak h'$. Let us consider the following objects:
\begin{itemize}\itemsep=0pt
\item[$(a)$] the quotient map $\pi\colon G \to M$ where $M$ is the variety of cosets $H\setminus G$, and $\nabla$ the Lie connection associated to
the parallelism $\pi_*\colon \mathfrak h' \to \mathfrak X[M]$ in $M$ $($as given in Example~{\rm \ref{ex:quotient_c})};
\item[$(b)$] its reciprocal Lie connection $\nabla^{\rm rec}$ on~$M$;
\item[$(c)$] the Lie algebras of right invariant vector fields
\begin{gather*}\mathfrak g^{\rm rec} = {\bf i}_*(\mathfrak g), \qquad \mathfrak h'^{\rm rec} = {\bf i}_*(\mathfrak h'),\end{gather*}
where ${\bf i}$ is the inverse map on $G$.
\end{itemize}
Then, the following statements are true:
\begin{itemize}\itemsep=0pt
\item[$(i)$] $\mathfrak h'$ is an ideal of $\mathfrak g$ $($equivalently $\mathfrak h'^{\rm rec}$ is an ideal of $\mathfrak g^{\rm rec})$;
\item[$(ii)$] $\mathfrak h$ is commutative $($equivalently $H$ is virtually abelian$)$;
\item[$(iii)$] the adjoint action of $G$ on $\mathfrak g^{\rm rec}$ preserves $\mathfrak h'^{\rm rec}$ and thus gives, by restriction, a morphism $\overline{\operatorname{Adj}}\colon G \to \operatorname{Aut}(\mathfrak h'^{\rm rec})$;
\item[$(iv)$] The Galois group of the connection $\nabla^{\rm rec}$ is ${\overline{\operatorname{Adj}}}(H) \subseteq \operatorname{Aut}(\mathfrak h'^{\rm rec})$ and thus virtually abelian.
\end{itemize}
\end{Proposition}

\begin{proof}We have that $\mathfrak g$ is the algebraic hull of $\mathfrak h'$. From Lemma~\ref{ap2_2} in Appendix~\ref{apB} we obtain $[\mathfrak g,\mathfrak g] \subseteq \mathfrak h'$. Statement~(i) follows straightforwardly. Let us consider $A$ and $B$ in $\mathfrak h$. Then $[A,B]$ is in $\mathfrak h$ and also in $\mathfrak h'$ by the previous argument. Thus, $[A,B]=0$ and this f\/inishes the proof of statement~(ii). Let us denote by $H'$ the subgroup of $G$ spanned by the image of $\mathfrak h'$ by the exponential map. For each element $h\in H'$, the adjoint action of $h$ preserves the Lie algebra~$\mathfrak h'$. By continuity of the adjoint action in the Zariski topology, we have that~$\mathfrak h'$ is preserved by the adjoint action of all elements of~$G$. This proves statement~(iii). In order to prove the last statement in the proposition we have to construct a Picard--Vessiot extension for the connec\-tion~$\nabla^{\rm rec}$. Let us consider a basis $\{A_1,\ldots, A_m\}$ of $\mathfrak h'$ and let $\bar A_i$ be the projection~$\pi_*(A_i)$. We have an extension of dif\/ferential f\/ields
\begin{gather*}(\mathbf C(M), \bar{\mathcal D}) \subseteq (\mathbf C(G), \mathcal D),\end{gather*}
where $\bar{\mathcal D}$ stands for the $\mathbf C(M)$-vector space of derivations spanned by $\bar A_1,\ldots, \bar A_m$ and $\mathcal D$ stands for the $\mathbf C(G)$-vector space of derivations spanned by $A_1,\ldots,A_m$ (see Appendix~\ref{ApA} for our conventions on dif\/ferential f\/ields).

The projection $\pi$ is a principal $H$-bundle. Any rational f\/irst integral of $\{A_1,\ldots,A_m\}$ is constant along $H'$ and thus it is necessarily a complex number. Thus, the above extension has no new constants and it is strongly normal in the sense of Kolchin, with Galois group $H$. Note that the dif\/ferential f\/ield automorphism corresponding to an element $h\in H$ is the pullback of functions by the left translation $L_{h}^{-1}$, that is, $(hf)(g) = f\big(h^{-1}g\big)$.

The horizontal sections for the connection $\nabla^{\rm rec}$ are characterized by the dif\/ferential equations
\begin{gather}\label{eq_proof_H}
[\bar A_i, X] = 0.
\end{gather}
Let us consider $\{B_1,\ldots,B_m\}$ a basis of $\mathfrak h'^{\rm rec}$. From the Zariski closedness of $H$ in~$G$ it follows that there are regular functions $f_{ij}\in\mathbf C[G]$ such that $B_i = \sum\limits_{j=1}^m f_{ij} A_j$. Thus let us def\/ine $\bar B_i = \sum\limits_{j=1}^m f_{ij} \bar A_j$. Those objects are vector f\/ields in $M$ with coef\/f\/icients in $\mathbf C[G]$, and clearly satisfy equation~\eqref{eq_proof_H}. Thus, the Picard--Vessiot extension of $\nabla^{\rm rec}$ is spanned by the functions~$f_{ij}$ and it is embedded, as a dif\/ferential f\/ield, in~$\mathbf C(G)$. Let us denote such extension by $\mathbf L$. We have a~chain of extensions
\begin{gather*}\mathbf C(M) \subseteq \mathbf L \subseteq \mathbf C(G).\end{gather*}
By Galois correspondence, the Galois group of $\nabla^{\rm rec}$ is a quotient $H/K$ where $K$ is the subgroup of elements of $H$ that f\/ix, by left translation, the functions $f_{ij}$. In order to prove statement~(iv) we need to check that this group $K$ is the kernel of the morphism $\overline{\operatorname{Adj}}$.

Let us note that the image under the adjoint action by $g\in G$ of an element $B\in \mathfrak h'^{\rm rec}$ is given by the left translation, $\overline{\operatorname{Adj}}(g)(B) = L_{g*}(B)$. This transformation makes sense for any derivation of $\mathbf C[G]$, and thus we have an action of~$G$ on $\mathfrak X(G)$. Let us take $h$ in the kernel of $\overline{\operatorname{Adj}}$, thus, $\overline{\operatorname{Adj}}(h)(B_j) = B_j$ for any index~$j$. Applying the transformation $L_{h*}$ to the expression of~$B_i$ as linear combination of the left invariant vector f\/ields~$A_j$ we obtain $B_i = \sum\limits_{j=1}^m L_{h*}(f_{ij} A_j) = \sum\limits_{j=1}^m h(f_{ij})A_j$. The coef\/f\/icients of~$B_i$ as linear combination of the $A_j$ are unique, and thus, $h(f_{ij}) = f_{ij}$ we conclude that~$h$ is an automorphism f\/ixing $\mathbf L$. On the other hand, let us take $h\in H$ f\/ixing $\mathbf L$. Then $L_{h*}\big(\sum f_{ij}A_j\big) = \sum f_{ij} A_j$ thus $\overline{\operatorname{Adj}}(h)(B_i) = B_i$ and then $h$ is in the kernel of $\overline{\operatorname{Adj}}$.
\end{proof}

\subsection[Some examples of $\mathfrak{sl}_2$-parallelisms]{Some examples of $\boldsymbol{\mathfrak{sl}_2}$-parallelisms}\label{sl2}

We will construct some parallelized varieties as subvarieties of the arc space of the af\/f\/ine line~${\mathbb A}^1_\mathbf C$. This family of
examples show how to realize every subgroup of ${\rm PSL}_2(\mathbf C)$ as the Galois group of the reciprocal Lie connection.

\subsubsection{The arc space of the af\/f\/ine line and its Cartan 1-form} \label{sl2parallelisms}

In our special case, the arc space of the af\/f\/ine line ${\mathbb A}^1_\mathbf C$ with af\/f\/ine coordinate $z$, is the space of all formal power series $\widehat{z} = \sum z^{(i)} \frac{x^i}{i !}$. It will be denoted by ${\mathscr{L}}$, its ring of regular functions is $\mathbf C[ {\mathscr{L}}] = \mathbf C\big[z^{(0)},z^{(1)}, z^{(2)},\ldots \big]$. For an open subset $U\subset \mathbf C$ one denotes by~${\mathscr{L}}U$ the set of power series $\widehat z$ with $z^{(0)} \in U$.

A biholomorphism $f\colon U \to V$ between open sets of $\mathbf C$ can be lift to a biholomorphism $ {\mathscr{L}}f\colon {\mathscr{L}}U \to {\mathscr{L}}V$ by composition $\widehat{z} \to f\circ\widehat{z}$.

Let $\widehat{\mathfrak X}$ be the Lie algebra of formal vector f\/ields $\mathbf C[[x]]\frac{\partial}{\partial x}$. One can build a rational form $\sigma\colon T{\mathscr{L}} \to \widehat{\mathfrak X}$ in following way (see \cite[Section~2]{guillemin-sternberg}). Let $v = \sum a_i \frac{\partial}{\partial z^{(i)}}$ be a tangent vector at the formal coordinate $\widehat{p}$, i.e., an arc in the Zariski open subset $\{z^{(1)} \not = 0\}$. The local coordinate $\widehat{p}$ can be used to have formal coordinates $p_0,p_{1}, p_{2}, \ldots $, on $\mathscr{L}$ and $v$ can be written $v = \sum b_i \frac{\partial}{\partial p_{i}}$. The form~$\sigma$ is def\/ined by $\sigma(v) = \sum b_i \frac{x^i}{i !} \frac{\partial}{\partial x} $. This form is rational and is an isomorphism between~$T_p {\mathscr{L}}$ and $\widehat{\mathfrak X}$ satisfying $d\sigma = - \frac{1}{2}[\sigma, \sigma]$ and $({\mathscr{L}}f)^\ast \sigma = \sigma$ for any biholomorphism~$f$.

This means that $\sigma$ provides an action of $\widehat{\mathfrak X}$ commuting with the lift of biholomorphisms. This form seems to be a~coparallelism but it is not compatible with the natural structure of pro-variety of~${\mathscr{L}}$ and~$\widehat{\mathfrak X}$: $\sigma^{-1}\big(\frac{\partial}{\partial x}\big) = \sum\limits_{i \geq 0} z^{(i+1)}\frac{\partial}{\partial z^{(i)}}$ is a derivation of degree $+1$ with respect to the pro-variety structure of ${\mathscr{L}}$. The total derivation above will be denoted by~$E_{-1}$. This gives a dif\/ferential structure to the ring~$\mathbf C[ {\mathscr{L}}]$.

\subsubsection{The parallelized varieties}

Let $\nu\in \mathbf C(z)$ be a rational function, $f$ be the rational function on the arc space given by the Schwarzian derivative
\begin{gather*} f\big(z^{(0)},z^{(1)},z^{(2)},z^{(3)}\big) = \frac{z^{(3)}}{z^{(1)}}-\frac{3}{2} \left(\frac{z^{(2)}}{z^{(1)}}\right) ^2+ \nu\big(z^{(0)}\big)\big(z^{(1)}\big)^2,\end{gather*}
and $I \subset \mathbf C[{\mathscr L}]$ be the $E_{-1}$-invariant ideal generated by $p\big(z^{(0)}\big) {z^{(1)}}^2f\big(z^{(0)},z^{(1)},z^{(2)},z^{(3)}\big)$ where $p$ is a minimal denominator of~$\nu$.

\begin{Lemma} The zero set $V$ of $I$ is a dimension $3$ subvariety of ${\mathscr{L}}$ and $\omega (TV) = {\mathfrak{sl}}_2(\mathbf C) \subset \widehat{\mathfrak X}$. This provides a ${\mathfrak{sl}}_2$-parallelism on~$V$.
\end{Lemma}

\begin{proof} One can compute explicitly this parallelism using $z^{(0)}$, $z^{(1)}$ and $z^{(2)}$ as \'etale coordinates on a Zariski open subset of $V$. Let us f\/irst compute the $\mathfrak{sl}_2$ action on ${\mathscr{L}}$. The standard inclusion of $\mathfrak{sl}_2$ in $\widehat{\mathfrak X}$ is given by $E_{-1} = \frac{\partial}{\partial x}$, $E_0 = x\frac{\partial}{\partial x}$ and $E_{1} = x^2\frac{\partial}{\partial x}$. Their actions on ${\mathscr{L}}$ are given by $E_{-1} = \sum z^{(i+1)}\frac{\partial}{\partial z^{(i)}}$, $E_0 = \sum iz^{(i)}\frac{\partial}{\partial z^{(i)}}$ and $E_1= \sum i(i-1) z^{(i-1)}\frac{\partial}{\partial z^{(i)}}$. The ideal $I$ is generated by the functions $E_{-1}^n\cdot f$. By def\/inition $E_{-1}\cdot f \in I$, a direct computation gives that $E_0\cdot f = 2f \in I$, $E_1\cdot f = 0 \in I$. The relations in $\mathfrak{sl}_{2}$ give that $E_{-1}\cdot I \subset I$, $E_0\cdot I \subset I$ and $E_1\cdot I \subset I$, i.e., the vector f\/ields $E_{-1}$, $E_0$ and $E_1$ are tangent to~$V$.
\end{proof}

Now parameterizing $V$ by $z^{(0)}$, $z^{(1)}$ and $z^{(2)}$ one gets
\begin{gather*}
 E_{-1}|_{\mathbf C ^3} = z^{(1)}\frac{\partial}{\partial z^{(0)}} + z^{(2)}\frac{\partial}{\partial z^{(1)}} + \left(-\nu\big(z^{(0)}\big)\big(z^{(1)}\big)^3 + \frac{3}{2}\frac{(z^{(2)})^2}{z^{(1)}}\right) \frac{\partial}{\partial z^{(2)}},\nonumber\\
 E_0|_{\mathbf C^3} = z^{(1)}\frac{\partial}{\partial z^{(1)}} + 2 z^{(2)}\frac{\partial}{\partial z^{(2)}}, \qquad E_1|_{\mathbf C^3} = 2 z^{(1)}\frac{\partial}{\partial z^{(2)}}.\label{parrallel}
\end{gather*}
They form a rational $\mathfrak{sl}_{2}$-parallelism on $\mathbf C^3$ depending on the choice of a rational function in one variable.

\subsubsection{Symmetries and the Galois group of the reciprocal connection}

\begin{Theorem}\label{thm_SL2}Any algebraic subgroup of ${\rm PSL}_{2}(\mathbf C)$ can be realized as the Galois group of the reciprocal connection of a~parallelism of~$\mathbf C^3$.
\end{Theorem}

\begin{proof}A direct computation shows that $z\mapsto \varphi(z)$ is an holomorphic function satisfying
\begin{gather*} \frac{\varphi'''}{\varphi'}-\frac{3}{2} \left(\frac{\varphi''}{\varphi'}\right)^2 + \nu(\varphi)(\varphi')^2 = \nu(z),
\end{gather*}
if and only if its prolongation ${\mathscr{L}}\varphi\colon \widehat{z} \mapsto \varphi (\widehat{z})$ on the space ${\mathscr{L}}$ preserves~$V$ and preserves each of the vector f\/ields~$E_{-1}$,~$E_0$ and~$E_1$.

Taking inf\/initesimal generators of this pseudogroup, one gets for any local analytic solution of the linear equation
\begin{gather}\label{lin}
a''' + 2 \nu a' +\nu 'a =0,
\end{gather}
a vector f\/ield $X = a(z)\frac{\partial}{\partial z}$ whose prolongation on ${\mathscr{L}}$ is
\begin{gather*}
{\mathscr{L}}X = a\big(z^{(0)}\big)\frac{\partial}{\partial z^{(0)}} + a'\big(z^{(0)}\big)z^{(1)}\frac{\partial}{\partial z^{(1)}} + \big( a''\big(z^{(0)}\big)\big(z^{(1)}\big)^2 + a'\big(z^{(0)}\big) z^{(2)}\big)\frac{\partial}{\partial z^{(2)}} + \cdots.
\end{gather*}
The equation (\ref{lin}) ensures that ${\mathscr{L}}{X}$ is tangent to $V$. The invariance of $\sigma$, $({\mathscr{L}}X)_{\ast}\sigma =0$, ensures that ${\mathscr{L}}X$ commutes with the ${\mathfrak{sl}}_{2}$-parallelism given above. This means that for any solution~$a$ of~(\ref{lin}) the vector f\/ield
\begin{gather*} a\big(z^{(0)}\big)\frac{\partial}{\partial z^{(0)}} + a'\big(z^{(0)}\big)z^{(1)}\frac{\partial}{\partial z^{(1)}}+ \big( a''\big(z^{(0)}\big)\big(z^{(1)}\big)^2 + a'\big(z^{(0)}\big) z^{(2)}\big)\frac{\partial}{\partial z^{(2)}}, \end{gather*}
commutes with $E_{-1}|_{\mathbf C ^3}$, $E_0|_{\mathbf C ^3}$ and $E_{1}|_{\mathbf C ^3}$.

Then the linear dif\/ferential system of f\/lat section for the reciprocal connection reduces to the linear equation~(\ref{lin}). This equation is the second symmetric power of $y'' = \nu(z)y$. If $G \subset {\rm SL}_2(\mathbf C)$ is the Galois group of $y'' = \nu(z)y$ then the image of its second symmetric power representation $s^2\colon G \to \operatorname{Sym}^2(\mathbf C^2)$ is the Galois group of~(\ref{lin}). The kernel of this representation is $\{{\rm Id}, -{\rm Id}\}$ then the Galois group of~(\ref{lin}) is an algebraic subgroup of ${\rm PSL}_{2}(\mathbf C)$.

Let us remark that, as it follows from its def\/inition, the Galois group of an equation contains the monodromy group. Moreover one can determine the monodromy group of classical dif\/ferential equations. Hypergeometric equations depend on three complex numbers $(a,b,c)$
\begin{gather*}
z(1-z) F'' + (c-(a+b+1)z)F' - abF = 0, 
\end{gather*}
and is equivalent to
\begin{gather*}y'' = \nu(\ell, n,m ; z) y,\end{gather*}
with \begin{gather*} \nu(\ell,m,n ;z) = \frac{\big(1-\ell^2\big)}{4z^2} +\frac{1-m^2}{4(1-z)^2} +\frac{1-\ell^2-m^2+n^2}{4 z(1-z)},\end{gather*}
and
\begin{gather*}F = z^{-c/2}(1-z)^{(c-a-b-1)/2} y, \qquad \ell = 1-c,\qquad m =c-a-b,\qquad n=a-b.\end{gather*}
These two equations have the same projectivized Galois group in ${\rm PGL}_2(\mathbf C)$. Any algebraic subgroup of ${\rm PGL}_2(\mathbf C)$ will be realized by an appropriate choice of $(a,b,c)$.
\subsubsection{The whole group}

For $a=b=1/2$, $c = 1$, the hypergeometric equation is the Picard--Fuchs equation of Legendre family. Its monodromy group is $\Gamma(2) \subset {\rm SL}_2(\mathbf Z)$ and is Zariski dense in~${\rm SL}_2(\mathbf C)$.

\subsubsection{The triangular subgroups}
For $b=0$ and $a=-1$ one can compute a basis of solutions of the equation: $1$ and $\int \big( \frac{1-z}{z} \big)^c {\rm d}z$. If $c$ is not rational, the Galois group is the group of invertible matrices $\left[\begin{smallmatrix} u & v\\ 0 & 1 \end{smallmatrix}\right]$. When $c$ is rational then $u$ must be a root of the unity of order the denominators of~$c$. When $c\in \mathbf Z$, the Galois group is the group of matrices $\left[\begin{smallmatrix} 1 & v\\ 0 & 1 \end{smallmatrix}\right]$.

For $b=0 $ and $c= a+1$ a basis of solution is given by $z^{-a}$ and $1$. Its Galois group is a~subgroup of the group of matrices $\left[\begin{smallmatrix} u& 0 \\ 0 & 1 \end{smallmatrix}\right]$. The parameter $a$ is rational if and only if it is a f\/inite subgroup.

\subsubsection{The dihedral subgroups}

For $c=1/2$ and $a+b=0$, a basis of solution is given by $\big(\sqrt{z}+\sqrt{1-z}\big)^a$ and $\big(\sqrt{z}-\sqrt{1-z}\big)^a$. The monodromy group is a~dihedral group in ${\rm GL}_2(\mathbf C)$ whose quotients give dihedral subgroups of ${\rm PGL}_2(\mathbf C)$.

\subsubsection{The tetrahedral subgroup}
This group is the monodromy group of hypergeometric equation for $\ell = 1/3$, $m= 1/2$ and $n=1/3$. A basis of solution is given by \begin{gather*}(z-1)^{-1/12}\Big(\sqrt{3}\big(z^{1/3}+1\big) \pm 2\sqrt{z^{2/3} + z^{1/3} +1} \Big)^{1/4}. \end{gather*}

\subsubsection{The octahedral subgroup}
This group is the monodromy group of hypergeometric equation for $\ell = 1/2$, $m= 1/3$ and $n=1/4$. A basis of solution is given by
\begin{gather*}
(z-1)^{-1/24}\Big[\sqrt{3} \big( \big(\sqrt{z}-1\big)^{1/3} +\big(\sqrt{z}+1\big)^{1/3}\big)^{1/3}\\
\qquad{} \pm 2\sqrt{\big(\sqrt{z}-1\big)^{2/3} + (z-1)^{1/3} +\big(\sqrt{z}+1\big)^{2/3}}\Big]^{1/4}.\\
\end{gather*}

\subsubsection{The icosaedral subgroup}
This group is the monodromy group of hypergeometric equation for $\ell = 1/2$, $m= 1/3$ and $n=1/5$. As icosahedral group is not solvable, the solution space is not described using formulas as simple as in preceding examples.
\end{proof}

\section{Darboux--Cartan connections}\label{section_DC}

\subsection{Connection of parallelism conjugations}\label{DC conjugation}
Let $\omega$ be a rational coparallelism on $M$ of type $\mathfrak g$ and $G$ an algebraic group with Lie algebra of left invariant vector f\/ields $\mathfrak g$ and Maurer--Cartan form $\theta$. Denote by $M^\star$ the open subset of~$M$ in wich $\omega$ is regular. We will study the contruction of conjugating maps between the paralle\-lisms~$(M,\omega)$ and $(G,\theta)$.

Let us consider the trivial principal bundle $\pi\colon P = G\times M \to M$. In this bundle we consider the action of $G$ by right translations $(g, x) * g' = (gg', x)$. Let $\Theta$ be the $\mathfrak g$-valued form $\Theta = \theta - \omega$ in~$P$.

\begin{Definition}\label{DCconnection} The kernel of $\Theta$ is a rational f\/lat invariant connection on the principal bundle $\pi\colon P \to M$. We call it the Darboux--Cartan connection of parallelism conjugations from~$(M,\omega)$ to~$(G,\theta)$.
\end{Definition}

The equation $\Theta = 0$ def\/ines a foliation on $P$ transversal to the f\/ibers at regular points of~$\omega$. The leaves of the foliation are the graphs of analytic parallelism conjugations from~$(M,\omega)$ to~$(G,\theta)$. By means of dif\/ferential Galois theory the Darboux--Cartan connection has a Galois group $\operatorname{Gal}(\Theta)$ with Lie algebra~$\mathfrak{gal}(\Theta)$. The following facts are direct consequences of the def\/inition of the Galois group:
\begin{itemize}\itemsep=0pt
\item[(a)] there is a regular covering map $c\colon (M^\star, \omega) \to (U,\theta)$ with $U$ an open subset of $G$, and $c^*(\theta) = \omega$ if and only if $\operatorname{Gal}(\Theta) = \{1\}$;
\item[(b)] there is a regular covering map $c\colon (M^\star, \omega) \to (U,q_*\theta)$ with $U$ an open subset of~$G/H$,~$H$ a group of f\/inite index, and $c^*(q_*\theta|_U) = \omega$ if and only if $\mathfrak{gal}(\Theta) = \{0\}$.
\end{itemize}
In any case, the necessary and suf\/f\/icient condition for $(M,\omega)$ and $(G,\theta)$ to be isogenous parallelized varieties is that $\mathfrak{gal}(\Theta) = \{0\}$.

\subsection{Darboux--Cartan connection and Picard--Vessiot}

Note that the coparallelism $\omega$ gives a rational trivialization of~$TM$ as the trivial bundle of f\/i\-ber~$\mathfrak g$. In $TM$ we have def\/ined the connection $\nabla^{\rm rec}$ whose horizontal vector f\/ields are the symmetries of the parallelism. On the other hand, $G$ acts in $\mathfrak g$ by means of the adjoint action. The Cartan-Darboux connection induces then a connection~$\nabla^{\rm adj}$ in the associated trivial bundle $\mathfrak g\times M$ of f\/iber~$\mathfrak g$.

\begin{Proposition}\label{adj_conjugation}The map
\begin{gather*}\tilde\omega\colon \ \big(TM, \nabla^{\rm rec}\big) \to \big(\mathfrak g \times M, \nabla^{\rm adj}\big), \qquad X_x \mapsto (\omega_x(X_x), x)\end{gather*}
is a birational conjugation of the linear connections $\nabla^{\rm rec}$ and $\nabla^{\rm adj}$.
\end{Proposition}

\begin{proof}It is clear that the map $\tilde\omega$ is birational. Let us consider $\{A_1,\ldots,A_m\}$ a basis of~$\mathfrak g$. Let $\rho\colon \mathfrak g\to \mathfrak X(M)$ be the parallelism associated to the parallelism $\omega$ and let us def\/ine $X_i = \rho(A_i)$. Then $\{X_1,\ldots, X_n\}$ is a rational frame in $M$ and the map $\tilde\omega$ conjugates the vector f\/ield $X_i$ with the constant section $A_i$ of the trivial bundle of f\/iber~$\mathfrak g$. By def\/inition of the reciprocal connection
\begin{gather*}\nabla^{\rm rec}_{X_i}X_j = [X_i,X_j].\end{gather*}
On the other hand, by def\/inition of the adjoint action and application of the covariant derivative as in equation~\eqref{covariant_associated}
of Appendix~\ref{ap7_associated} we obtain
\begin{gather*}\nabla^{\rm adj}_{X_i}A_j = [A_i,A_j].\end{gather*}
Therefore we have that $\tilde\omega$ is a rational morphism of linear connections that conjugates $\nabla^{\rm rec}$ with~$\nabla^{\rm adj}$.
\end{proof}

The following facts follow directly from Proposition~\ref{adj_conjugation}, and basic properties of the Galois group.

\begin{Corollary}\label{cor:rec} Let us consider the adjoint action $\rm{ Adj}\colon G \to {\rm GL}(\mathfrak g)$ and its derivative ${\rm adj}\colon \mathfrak{g}\to \operatorname{End}(\mathfrak g)$. The following facts hold:
\begin{itemize}\itemsep=0pt
\item[$(a)$] $\operatorname{Gal}(\nabla^{\rm rec}) = \operatorname{Adj}({\rm Gal(\Theta)})$;
\item[$(b)$] $\mathfrak{gal}(\nabla^{\rm rec}) = {\rm adj}(\mathfrak{gal}(\Theta))$;
\item[$(c)$] if $\mathfrak g$ is centerless then $\mathfrak{gal}(\nabla^{\rm rec})$ is isomorphic to $\mathfrak{gal}(\Theta)$;
\item[$(d)$] assume $\mathfrak g$ is centerless, then the necessary and sufficient condition for $(M,\omega)$ and $(G,\theta)$ to be isogenous is that $\mathfrak{gal}(\nabla^{\rm rec})=\{0\}$.
\end{itemize}
\end{Corollary}

\begin{proof}(a) and (b). First, by Proposition \ref{adj_conjugation} we have that $\operatorname{Gal}(\nabla^{\rm rec}) = \operatorname{Gal}(\nabla^{\rm adj})$ and so $\mathfrak{gal}(\nabla^{\rm rec}) = \mathfrak{gal}(\nabla^{\rm adj})$. By def\/inition $\nabla^{\rm adj}$ is the associated connection induced by $\Theta$ in the associated bundle $\mathfrak g\times M$. This trivial bundle is the associated bundle induced by the adjoint representation $\operatorname{Adj}\colon G\to \operatorname{End}(\mathfrak g)$. Then, by Theorem~\ref{associated galois}, we have $\operatorname{Gal}(\nabla^{\rm rec}) = \operatorname{Adj}({\rm Gal(\Theta)})$ and $\operatorname{Gal}(\nabla^{\rm rec}) = \operatorname{Adj}({\rm Gal(\Theta)})$.

(c) It is a direct consequence of (b). The kernel of ${\rm adj}\colon \mathfrak g \to \operatorname{End}(\mathfrak g)$ is the center of~$\mathfrak g$.

(d) It follows from the def\/inition of Darboux--Cartan connection (see remarks after Def\/i\-ni\-tion~\ref{DCconnection}) that the necessary and suf\/f\/icient condition for $(M,\omega)$ and $(G,\theta)$ to be isogenous is that $\mathfrak{gal}(\nabla^{\rm rec})=\{0\}$. By point~(b) we conclude.
\end{proof}

\subsection{Algebraic Lie algebras}

Let us consider $(M,\omega)$ a rational coparallelism of type $\mathfrak g$ with $\mathfrak g$ a centerless Lie algebra. We do not assume \emph{a priori} that $\mathfrak g$ is an algebraic Lie algebra. The connection $\nabla^{\rm rec}$ is, as said in Proposition~\ref{adj_conjugation}, conjugated to the connection in $\mathfrak g\times M$ induced by the adjoint action. Note that, in order to def\/ine this connection we do not need the group operation but just the Lie bracket in~$\mathfrak g$. We have an exact sequence
\begin{gather*}0\to \mathfrak g' \to \mathfrak g \to \mathfrak g^{ab}\to 0,\end{gather*}
where $\mathfrak g'$ is the derived algebra $[\mathfrak g, \mathfrak g]$. Since the Galois group acts by adjoint action, we have that $\mathfrak g'\times M$ is stabilized by the connection $\nabla^{\rm rec}$ and thus we have an exact sequence of connections
\begin{gather*}0\to (\mathfrak g'\times M, \nabla')\to \big(\mathfrak g \times M,\nabla^{\rm rec}\big)\to \big(\mathfrak g^{ab}\times M, \nabla^{ab}\big)\to 0.\end{gather*}

\begin{Lemma} The Galois group of $\nabla^{ab}$ is the identity, therefore $\nabla^{ab}$ has a basis of rational horizontal sections.
\end{Lemma}

\begin{proof} By def\/inition, the action of $\mathfrak g$ in $\mathfrak g^{ab}$ vanishes. Thus, the constant functions $M\to \mathfrak g^{ab}$ are rational horizontal sections.
\end{proof}

\begin{Lemma}\label{th:algebraic} Let $\omega$ be a rational coparallelism of $M$ of type $\mathfrak g$ with $\mathfrak g$ a centerless Lie algebra.
If $\mathfrak{gal}(\nabla^{\rm rec}) = \{0\}$ then $\mathfrak g$ is an algebraic Lie algebra.
\end{Lemma}

\begin{proof} Assume $\mathfrak g$ is a linear Lie algebra and et $E$ be the smallest algebraic subgroup such that $ \operatorname{Lie} (E) = \mathfrak e\supset \mathfrak g$. We may assume that $E$ is also centerless. Let $A_{1}, \ldots, A_{r}$ be a basis of $\mathfrak g$, for $i=1,\ldots,r$, $X_{i} = \omega^{-1}(A_{i})$. Complete with $B_{1}, \ldots, B_{p}$ in such way that $A_1,\ldots,A_r,B_1,\ldots,B_p$ is a basis of $\mathfrak e$. We consider in $E \times M$ the distribution spanned by the vector f\/ields $A_{i}+X_{i}$. This is a $E$-principal connection called~$\nabla$.

Let $\overline{\nabla}$ be the induced connection via the adjoint representation on $\mathfrak e \times M$ then
\begin{enumerate}\itemsep=0pt
\item[1)] $\overline{\nabla}$ preserves $\mathfrak g$ and $\overline{\nabla}|_{\mathfrak g} = \nabla^{\rm rec}$, by hypothesis $\mathfrak{gal}(\overline{\nabla}|_{\mathfrak g}) = \{0\}$;
\item[2)] if $\widetilde{\nabla}$ is the quotient connection on ${\mathfrak e}/ \mathfrak g$ then $\mathfrak{gal}(\widetilde{\nabla}) =\{0\}$.
\end{enumerate}
If $\varphi \in \mathfrak{gal}(\overline{\nabla})$ then for any $X \in \mathfrak g$, $[X,B_{i}] \in \mathfrak g$ thus $0 = \varphi [X,B_{i}] = [ X, \varphi B_{i}]$ and $\varphi B_i$ commute with~$\mathfrak g$. From the second point above $\varphi B \in \mathfrak g$. By hypothesis $\varphi B_i =0$ and $\mathfrak{gal}(\overline{\nabla}) =\{0\}$. The projection on $E$ of an algebraic leaf of $\nabla$ gives an algebraic leaf for the foliation of~$E$ by the left translation by~$\mathfrak g$. This proves the lemma.
\end{proof}

\begin{Theorem}\label{th_criteria} Let $\mathfrak g$ be a centerless Lie algebra. An algebraic variety $(M,\omega)$ with a rational parallelism of type~$\mathfrak g$ is isogenous to an algebraic group if and only if $\mathfrak{gal}(\nabla^{\rm rec}) = \{0\}$.
\end{Theorem}

\begin{proof} It follows directly from Lemma \ref{th:algebraic} and Corollary \ref{cor:rec}.
\end{proof}

\begin{Corollary}\label{th:pair}Let $\mathfrak g$ be a centerless Lie algebra. Any algebraic variety endowed with a pair of commuting rational parallelisms of type $\mathfrak g$ is isogenous to an algebraic group endowed with its two canonical parallelisms of left and right invariant vector fields.
\end{Corollary}

\begin{proof}Just note that to have a pair of commuting parallelism is a more restrictive condition than having a parallelism with vanishing Lie algebra of the Galois group of its reciprocal connection.
\end{proof}

This result can be seen as an algebraic version of Wang result in \cite{Wang}. It gives the classif\/ication of algebraic varieties endowed with pairs of commuting parallelisms. Assuming that the Lie algebra is centerless is not a superf\/luous hypothesis, note that the result clearly does not hold for abelian Lie algebras. There are rational $1$-forms in~$\mathbf{CP}_1$ that are not exact (isogenous to $(\mathbf C, {\rm d}z)$) nor logarithmic (isogenous to~$(\mathbf C^*,{\rm d}\log(z))$). In these examples, the pair of commuting parallelisms is given by twice the same parallelism.

\begin{Remark}Let $(M,\omega,\omega')$ be a manifold endowed with a pair of commuting parallelism forms of type $\mathfrak g$, a centerless Lie algebra. From Lemma~\ref{th:algebraic} we have that $\mathfrak g$ is an~algebraic Lie algebra. We can construct the algebraic group enveloping $\mathfrak g$ as follows. We consider the adjoint action
\begin{gather*}{\rm adj}\colon \ \mathfrak g\hookrightarrow \operatorname{End}(\mathfrak g).\end{gather*}
The algebraic group enveloping $\mathfrak g$ is identif\/ied with the algebraic subgroup $G$ of $\operatorname{Aut}(\mathfrak g)$ whose Lie algebra is ${\rm adj}(\mathfrak g)$. From Corollary~\ref{cor:rec}(a), we have that $\operatorname{Gal}(\Theta)=\{e\}$. Thus, there is a~rational map $f\colon M\to G$ such that $f^*(\theta) = \omega$, where~$\theta$ is the Maurer--Cartan form of~$G$. We can express explicitly this map in terms of the commuting parallelism forms. For each $x\in M$ in the domain of regularity of the parallelisms, $\omega(x)$ and $\omega'(x)$ are isomorphisms of~$T_xM$ with~$\mathfrak g$. We def\/ine
\begin{gather*}f(x) = - \omega(x)\circ \omega'(x)^{-1}.\end{gather*}
\end{Remark}

\begin{Remark}In virtue of Corollary~\ref{th:pair}, if $\mathfrak g$ is a non-algebraic centerless Lie algebra, there is no algebraic variety endowed with a pair of regular commuting parallelisms of type $\mathfrak g$. This limits the possible generalizations of Theorem~\ref{TDeligne}.
\end{Remark}

\begin{Remark}B.~Malgrange has given in~\cite{malgrange-P} another criterion: If $(M, \omega)$ is a parallelized variety and $\mathcal F$ is the foliation on $M \times M$ given by $\operatorname{pr}_1^\ast \omega - \operatorname{pr}_2^\ast \omega = 0$. Then $(M,\omega)$ is birational to an algebraic group if and only if leaves of~$\mathcal F$ are graphs of rational maps. The relations with Theorem~\ref{th_criteria} and Corollary~\ref{th:pair} are the following. One can identify $TM$ with the vertical tangent (i.e., the kernel of~${\rm d} \operatorname{pr}_2$) along the diagonal in $M\times M$. The diagonal is a leaf of $\mathcal{F}$ and the linearization of $\mathcal F$ along the diagonal def\/ines a connection $\nabla_{\mathcal F}$ on $TM$. By construction:
\begin{itemize}\itemsep=0pt
\item $\nabla_{\mathcal F}$-horizontal sections commute with the parallelism, it is the reciprocal Lie connection;
\item if leaves of $\mathcal F$ are algebraic then $\nabla_{\mathcal F}$-horizontal section are algebraic.
\end{itemize}
\end{Remark}

\section{Some homogeneous varieties}

The notion of {\it isogeny} can be extended beyond the simply-transitive case. Let us consider a~complex Lie algebra $\mathfrak g$. An {\it infinitesimally homogeneous variety} of type $\mathfrak g$ is a pair $(M,\mathfrak s)$ consisting of a complex smooth irreducible variety $M$ and a f\/inite-dimensional Lie algebra \smash{$\mathfrak s \subset \mathfrak X(M)$} isomorphic to~$\mathfrak g$.

As before, we are interested in conjugation by rational and algebraic maps so that, whenever necessary, we replace $M$ by a suitable Zariski open subset. In this context, we say that a~dominant rational map $f\colon M_1 \dasharrow M_2$ between varieties of the same dimension conjugates the inf\/initesimally homogeneous varieties $(M_1,\mathfrak s_1)$ and $(M_2,\mathfrak s_2)$ if $f^*(\mathfrak s_2) = \mathfrak s_1$. We say that $(M_1,\mathfrak s_1)$ and $(M_2,\mathfrak s_2)$ are {\it isogenous} if they are conjugated to the same inf\/initesimally homogeneous space of type $\mathfrak g$.

Let $G$ be an algebraic group over $\mathbf C$, $K$ an algebraic subgroup, $\mathfrak{lie}(G)$ its Lie algebra of left invariant vector f\/ields and $\mathfrak{lie}(G)^{\rm rec}$ its Lie algebra of right invariant vector f\/ields. A natural example of inf\/initesimally homogeneous space are the homogeneous spaces $G/H$ endowed with the induced action of the Lie algebra $\mathfrak{lie}(G)^{\rm rec}$. We want to recognize when a~inf\/initesimally homogeneous space is isogenous to an homogeneous space. We prove that if $\mathfrak s \subset \mathfrak X(M)$ is a {\it normal} Lie algebra of vector f\/ields then~$(M,\mathfrak s)$ is isogenous to a homogeneous space. In particular, we prove that any $n$-dimensional inf\/initesimally homogeneous space of type $\mathfrak{sl}_{n+1}(\mathbf C)$ is isogenous to the projective space. Our answer is based on a generalization of the computations done in Section~\ref{sl2}.

\subsection[The $\mathfrak{sl}_2$ case]{The $\boldsymbol{\mathfrak{sl}_2}$ case}

\begin{Theorem}[Loray--Pereira--Touzet (private communication)] Let $\Cc$ be a curve with~$X$,~$Y$,~$H$ three rational vector f\/ields such that $[X,Y] = H$, $[H,X] = -X$ and $[H,Y] = Y$. Then there exists a rational dominant map $h \colon \Cc \dasharrow \mathbf {CP}_1$ such that $X = h^\ast\big(\frac{\partial}{\partial z}\big)$, $H = h^\ast \big(z\frac{\partial}{\partial z}\big)$ and $Y = h^\ast \big(z^2\frac{\partial}{\partial z}\big)$.
\end{Theorem}
Their proof is elementary. We outline here a more sophisticated proof in the case $\Cc = A^1_{\mathbf C}$ that will be generalized in the next section.
\begin{proof} Notations are the ones introduced in Section~\ref{sl2}. $\Lc$ is the space of parameterized arcs $\widehat{z} = \sum_i z^{(i)} \frac{x^i}{i!}$ on $\Cc$. The vector space $\mathbf C X + \mathbf C H + \mathbf C Y$ is denoted by $\mathfrak g$. Let $r_o\colon (\mathbf C, 0) \to A^1_\mathbf C$ be an arc with $r_o'(0) \not = 0$ and consider $V \subset \Lc$ def\/ined by
\begin{gather*}
V =\{ \widehat{z} \in \Lc \, | \, \widehat{z}^\ast \mathfrak g =r_o^\ast \mathfrak g \}.
\end{gather*}
\begin{Claim}
This is a $3$-dimensional algebraic variety.
\end{Claim}

\begin{Claim}The prolongations $\Lc X$, $\Lc Y$ and $\Lc H$ define a $\mathfrak{sl}_2$-parallelism on $V$.
\end{Claim}

Let us describe the canonical structure of $\Lc$ (see \cite[pp.~11--12]{guillemin-sternberg} or next section for a dif\/ferent presentation).
For $k$ an integer greater or equal to~$-1$, let us consider the vector f\/ield on~$\Lc$
\begin{gather*}E_k = \sum_{i\geq k} \frac{i !}{(i-k-1)!} z^{(i-k)}\frac{\partial}{\partial z^{(i)}}.\end{gather*}
We def\/ine a morphism of Lie algebra $\rho\colon \widehat{\mathfrak X} \to \mathfrak X(\Lc)$ by $x^{k+1}\frac{\partial}{\partial x} \mapsto E_k$ and the adic continuity.

\begin{Claim}The Cartan form $\sigma$ $($as defined in Section~{\rm \ref{sl2parallelisms})} restricted to $V$ takes values in the Lie algebra $r_0^*(\mathfrak g)$. It is the parallelism form reciprocal to the parallelism~$\Lc X$,~$\Lc H$ and~$\Lc Y$ of~$V$.
\end{Claim}

Using Corollary \ref{th:pair}, $V$ is isogeneous to ${\rm PSL}_2(\mathbf C)$ as def\/ined in Def\/inition~\ref{isogenous}. For $p\in M$, $V_p = \{ \widehat{z} \in V \ |\ \widehat{z}(0)=p\}$ are homogeneous spaces for the action of $\widetilde{K} = \{\varphi \colon (\mathbf C,0) \to (\mathbf C,0) \, |\, r_o\circ \varphi \in V\}$, i.e., $\Cc = V/\widetilde{K}$. Let $K$ be the subgroup of ${\rm PSL}_2(\mathbf C)$ of upper triangular matrices.

\begin{Claim}The actions of $\widetilde{K}$ on $V$ and the right action of $K$ on ${\rm PSL}_2(\mathbf C)$ are conjugated by the isogeny.
\end{Claim}

This induces an isogeny between $\Cc$ and $\mathbf {CP}_1$. Let $\pi_1$ and $\pi_2$ be the two maps of the isogeny. A local transformation $\varphi$ such that $\pi_1\circ \varphi = \pi_1$ satisf\/ies $\varphi^\ast \pi_1^\ast(X,H,Y) = \pi_1^\ast(X,H,Y)$ and the same is true for the push-forward $(\pi_2)_\ast \varphi$ of $\varphi$ on $\mathbf {CP}_1$. Then $(\pi_2)_\ast \varphi$ preserves $\frac{\partial}{\partial z}$ and $z\frac{\partial}{\partial z}$. It is the identity. This f\/inishes the proof.
\end{proof}

\subsection{Some jet spaces}\label{some_jet_spaces}

Let $M$ be a $n$-dimensional af\/f\/ine variety. The space of parameterized subspaces of $M$ is the set of formal maps: $ M^{[n]} = \{ r\colon (\mathbf C^n,0) \to M \}$. Like the arc space, it has a natural structure of pro-algebraic variety. We will give the construction of its coordinate ring following \cite[Section~2.3.2, p.~80]{beilinson-drinfeld}. Let $\mathbf C[\partial_1, \ldots, \partial_n]$ be the $\mathbf C$-vector space of linear partial dif\/ferential operators with constant coef\/f\/icients. The coordinate ring of $M^{[n]}$ is $\operatorname{Sym}(\mathbf C[M]\otimes \mathbf C[\partial_1,\ldots,\partial_n]) / \mathcal L$ where
\begin{itemize}\itemsep=0pt
 \item the tensor product is a tensor product of $\mathbf C$-vector spaces;
 \item $\operatorname{Sym}( V )$ is the $\mathbf C$-algebra generated by the vector space $V$;
 \item $\mathbf C[M]\otimes \mathbf C[\partial_1,\ldots, \partial_n]$ has a structure of $\mathbf C[\partial_1,\ldots,\partial_n]$-module {\it via}
 the right composition of dif\/ferential operators;
 \item $\operatorname{Sym}(\mathbf C[M]\otimes \mathbf C[\partial_1,\ldots, \partial_n])$ has the induced structure of $\mathbf C[\partial_1,\ldots,\partial_n]$-algebra;
 \item the Leibniz ideal $\mathcal L$ is the $\mathbf C[\partial_1,\ldots, \partial_n]$-ideal generated by $fg\otimes 1 - (f\otimes1)(g\otimes1)$ for all $(f,g) \in \mathbf C[M]^2$ and by $1 - 1\otimes 1$.
 \end{itemize}
Local coordinates $(z_1, \ldots, z_n)$ on $M$ induce local coordinates on $M^{[n]}$ {\it via} the Taylor expansion of maps $r$ at $0$
\begin{gather*}
r(x_1\ldots, x_n) = \left( \sum_{\alpha \in \mathbf N^n} r_1^{\alpha} \frac{x^\alpha}{\alpha!}, \ldots,\sum_{\alpha \in \mathbf N^n} r_n^{\alpha} \frac{x^\alpha}{\alpha!} \right).
\end{gather*}
One denotes by $z_i^{\alpha}\colon M^{[n]} \to \mathbf C$ the function def\/ined by $z_i^{\alpha}(r) = r_i^{\alpha}$. This function is the element $z_i\otimes \partial^\alpha$ in~$\mathbf C[M^{[n]}]$.

\subsubsection{Prolongation of vector f\/ields}
Any derivation $Y$ of $\mathbf C[M]$ can be trivially extended to a derivation of $\operatorname{Sym}(\mathbf C[M]\otimes \mathbf C[\partial_1, \ldots, \partial_n])$. It preserves the ideal generated by $fg\otimes 1 - (f\otimes1)(g\otimes1)$ for all $(f,g) \in \mathbf C[M]^2$ and by $1 - 1\otimes 1$ and commutes with the action of $\mathbf C[\partial_1, \ldots, \partial_n]$ then it preserves the Leibniz ideal and def\/ines a~derivation of $\mathbf C[M^{[n]}]$. This derivation is called the prolongation of $Y$, and it is denoted by~$Y^{[n]}$.

The same procedure can be used to def\/ine the prolongation of analytic or formal vector f\/ields on~$M$ to~$M^{[n]}$.

\subsubsection{The canonical structure}\label{canonicalst}

The jet space $M^{[n]}$ is endowed with a dif\/ferential structure on its coordinate ring and with a group action by ``reparameterizations''. The compatibility condition between these two structures is well-known (see \cite[pp.~11--23]{guillemin-sternberg}) and is easily obtained using the construction above.

The action of $\partial_j\colon \mathbf C[M^{[n]}] \to \mathbf C[M^{[n]}]$ can be written in local coordinates and gives the total derivative operator $\sum_{i,\alpha} z_i^{\alpha + 1_j} \frac{\partial}{\partial z_i^{\alpha}}$. It is the dif\/ferential structure of the jet space. The pro-algebraic group \begin{gather*}\Gamma = \big\{ \gamma\colon (\mathbf C^n,0) \overset{\sim}{\rightarrow} (\mathbf C^n,0); \text{ formal invertible}\big\}\end{gather*}
 acts on $M^{[n]}$.This action is denoted by $S \gamma (r) = r\circ \gamma$.

These two actions arise from the action of the Lie algebra $\widehat{\mathfrak{X}} = \bigoplus \mathbf C[[x_1,\ldots,x_n]]\partial_i$ on $M^{[n]}$. This action is described on the coordinate ring in the following way. For $\xi \in \widehat{\mathfrak{X}}$, $f\in \mathbf C[M]$ and $P \in \mathbf C[\partial_1,\ldots,\partial_n]$, we def\/ine $\xi \cdot ( f \otimes P) = f\otimes (P\circ \xi)|_0$ where the composition is evaluated in~$0$ in order to get an element of $\mathbf C[\partial_1,\ldots,\partial_n]$. The action of $\bigoplus \mathbf C \partial_i$ is the dif\/ferential structure. The action of $\widehat{\mathfrak{X}}^0 = \mathfrak{lie}(\Gamma)$, the Lie subalgebra of vector f\/ields vanishing at $0$ is the inf\/initesimal part of the action of $\Gamma$.

 \begin{Theorem}[\cite{guillemin-sternberg}]\label{canonique} Let $M^{[n]\ast}$ be the open subset of submersions. The action above gives a~canonical form $\sigma\colon T M^{[n]\ast} \to \widehat{\mathfrak{X}}$ satisfying:
\begin{itemize}\itemsep=0pt
\item for any $r \in M^{[n]\ast}$, $\sigma$ is a isomorphism from $T_r M^{[n]\ast}$ to $\widehat{\mathfrak{X}}$;
\item for any $\gamma \in \Gamma$, $(S\gamma)^\ast \sigma = \gamma^\ast \circ \sigma$;
\item $d\sigma = -\frac{1}{2}[\sigma, \sigma]$.
\end{itemize}
 \end{Theorem}
 These equalities are {\it not} compatible with the projective systems.

\subsection{Normal Lie algebras of vectors f\/ields}

Without lost of generality, we should
\begin{enumerate}\itemsep=0pt
\item[1)] identify $\mathfrak g$ with its image in $\mathfrak X(M)$;
\item[2)] replace $M$ by a Zariski open subvariety on which $\mathfrak g$ is def\/ined and of maximal rank at any point.
\end{enumerate}
If $p\in M$ one can identify $\mathfrak g$ with a Lie subalgebra of $\widehat{\mathfrak X}(M,p)$, the Lie algebra of formal vector f\/ields on~$M$ at~$p$.

\begin{Definition} For a Lie subalgebra $\mathfrak g \subset \mathfrak X[M]$, its normalizer at $p\in M$ is \begin{gather*} \widehat{N}(\mathfrak g,p) = \big\{ Y \in \widehat{\mathfrak X}(M,p) \, |\, Y,\mathfrak g] \subset \mathfrak g\big\}.\end{gather*}
\end{Definition}

\begin{Definition}A Lie subalgebra $\mathfrak g \subset \mathfrak X[M]$ is said to be normal if for generic $p \in M$ on has $ \widehat{N}(\mathfrak g,p) = \mathfrak g$.
\end{Definition}

\begin{Lemma} If $\mathfrak g$ is transitive then the Lie algebra $ \widehat{N}(\mathfrak g,p) $ is finite-dimensional.
\end{Lemma}
\begin{proof} Let $k$ be an integer large enough so that the only element of $\mathfrak g$ vanishing at order~$k$ at~$p$ is~$0$. If $\widehat{N}(\mathfrak g,p) $ is not f\/inite-dimensional then there exists a non-zero $Y \in \widehat{N}(\mathfrak g,p)$ vanishing at order $k+1$ at $p$. For $X \in \mathfrak g$, the Lie bracket $[Y,X]$ is an element of $\mathfrak g$ vanishing at order~$k$ at~$p$. It is zero meaning that $Y$ is invariant under the f\/lows of vector f\/ields in~$\mathfrak g$. The transitivity hypothesis together with $Y(p)=0$ proves the lemma.
 \end{proof}
\begin{Lemma}If there exists a point $p\in M$ such that $\mathfrak g$ is maximal among finite-dimensional Lie subalgebra of~$\widehat{\mathfrak X}(M,p)$ then $\mathfrak g$ is normal.
\end{Lemma}
\begin{proof}
Because of the preceding lemma, if such a point exists then $\mathfrak g = \widehat{N}(\mathfrak g,p)$ in $\widehat{\mathfrak X}(M,p)$. By transitivity, for any couple of points $(p_1,p_2) \in M^2$ there is a composition of f\/lows of elements of~$\mathfrak g$ sending~$p_1$ on~$p_2$. These f\/lows preserve $\mathfrak g$ thus the equality holds at any~$p$.
\end{proof}

\begin{Example} Let $M$ be $n$-dimensional and $\mathfrak g$ be a transitive Lie subalgebra of rational vector f\/ields isomorphic to $\mathfrak{sl}_{n+1}(\mathbf C)$. Then $\mathfrak g$ is normal (see~\cite{cartan}).
\end{Example}

\subsection[Centerless, transitive and normal $\Rightarrow$ isogenous to a homogeneous space]{Centerless, transitive and normal $\boldsymbol{\Rightarrow}$ isogenous to a homogeneous space}

\begin{Theorem}\label{homogeneous} Let $M$ be a smooth irreducible algebraic variety over $\mathbf C$ and $\mathfrak g$ be a transitive, centerless, normal, finite-dimensional Lie subalgebra of $\mathfrak X(M)$. Then there exists an algebraic group $G$, an algebraic subgroup $H \subset G$ and an isogeny between $(M,\mathfrak g)$ and $(G/H, \mathfrak{lie}(G))$. Moreover, if $N_G(\mathfrak{lie}(H)) = H$ then the isogeny is a dominant rational map~$M \dasharrow G/H$.
\end{Theorem}

Because of the f\/initeness and the transitivity, there exists an integer~$k$ such that at any $p \in M$ and for any $Y \in \widehat{N}(\mathfrak g,p)$, $j_k(Y)(p) \not = 0$, unless $Y=0$.

Let $r_o\colon (\mathbf C^n,0) \to M$ be an invertible formal map with $r_o(0) =p$ a regular point. Let us consider the subspace of $M^{[n]}$ def\/ined by
\begin{gather*}V = \{ r \colon (\mathbf C^n,0) \to M \, |\, r^\ast \mathfrak g = r_o^\ast \mathfrak g\}.\end{gather*}

\begin{Lemma} $V$ is finite-dimensional. \end{Lemma}
\begin{proof} If $r_o^{-1}\circ r$ is tangent to the identity at order $k$ then the induced automorphism of $\mathfrak g$ is the identity. The map $r_o^{-1}\circ r$ f\/ixes $p$, thus it is the identity. This proves the lemma.
\end{proof}

Using $r_o$ one can identify the Lie algebra $\widehat{N}(\mathfrak g,p)$ with a Lie subalgebra of $\widehat{\mathfrak{X}}$. The latter acts on $M^{[n]}$ as described in Section~\ref{canonicalst}. As an application of the Theorem~\ref{canonique}, one gets:

\begin{Lemma} The restriction of the canonical structure of $M^{[n]}$ gives an parallelism
\begin{gather*} TV = r_o^\ast(\widehat{N}(\mathfrak g,p)) \times V,\end{gather*} called the canonical parallelism.
\end{Lemma}

\begin{Lemma}The horizontal sections of the reciprocal Lie connection of the canonical parallelism are~$Y^{[n]}$ for $Y \in \widehat{N}(\mathfrak g,q)$ for $q\in M$.
\end{Lemma}

\begin{Lemma}Under the hypothesis of normality of $\mathfrak g$, $V$ has two commuting parallelisms of type~$\mathfrak g$.
\end{Lemma}

Using Corollary \ref{th:pair}, $\mathfrak g$ is the Lie algebra of an algebraic group $G$ isogeneous to $V$. $V$ is foliated by the orbits of the subgroup $K$ of $\Gamma$ stabilizing $V$. This group is algebraic with Lie algebra $\mathfrak k = r_o^\ast (\mathfrak g) \cap \widehat{\mathfrak{X}}^0$. Let $\mathfrak h \subset \mathfrak{lie}(G)$ be the Lie algebra corresponding to $\mathfrak k$ by the isogeny. Then the orbits of $\mathfrak h$ are algebraic. This means that~$\mathfrak h$ is the Lie algebra of an algebraic subgroup~$H$ of~$G$, and that~$V/K$ and~$G/H$ are isogenous.

Assume that $N_{G}(\mathfrak{lie}(H)) = H$. If $W$ is the isogeny between $V$ and $G$. The push-forward of a~local analytic deck transformation of $W \to V$ is a transformation of~$G$ preserving each element of~$\mathfrak g$, it is a right translation. A deck transformation preserves the orbits of the pull-back of~$\mathfrak k$ on~$W$. Its push-forward preserves the orbits of a group containing~$H$ with the same Lie algebra. By hypothesis the push-forward is in~$H$ and then the isogony obtained by taking the quotient under~$K$ and~$H$ is the graph of a dominant rational map.

\appendix

\section{Picard--Vessiot theory of a principal connection}\label{ApA}

In the previous reasoning we have used the concept of dif\/ferential Galois group of a connection. Here, we present a dictionary between invariant connection and strongly normal dif\/ferential f\/ield extension (in the sense of Kolchin). In our setting a dif\/ferential f\/ield is a pair $(\mathcal K, \mathcal D)$ where~$\mathcal K$ is a~f\/initely generated f\/ield over $\mathbf C$ and $\mathcal D$ is a $\mathcal K$ vector space of derivations of $\mathcal K$ stable by Lie bracket. The dimension of $\mathcal D$ is called the rank of the dif\/ferential f\/ield. Note that we can adapt this notion easily to that of a f\/inite number of commuting derivations by taking a~suitable basis of~$\mathcal D$. However we prefer to consider the whole space of derivations. With our def\/inition a~dif\/ferential f\/ield extension $(\mathcal K, \mathcal D) \to (\mathcal K', \mathcal D')$ is a f\/ield extension $\mathcal K \subset \mathcal K'$ such that each element of $\mathcal D$ extends to a~unique element of $\mathcal D'$ and such extensions span the space $\mathcal D'$ as $\mathcal K'$-vector space.

\subsection{Dif\/ferential f\/ield extensions and foliated varieties}

First, let us see that there is a natural dictionary between f\/initely generated dif\/ferential f\/ields over $\mathbf C$ and irreducible foliated varieties over $\mathbf C$ modulo birational equivalence. Let $(M,\mathcal F)$ be an irreducible foliated variety of dimension~$n$. The distribution $T\mathcal F \subset TM$ is of rank $r\leq n$. We denote by $\mathfrak X_{\mathcal F}$ the space of rational vector f\/ields in $T\mathcal F$; it is a $\mathbf C(M)$-Lie algebra of dimension~$r$. Hence, the pair $(\mathbf C(M),\mathfrak X_{\mathcal F})$ is a dif\/ferential f\/ield. The f\/ield of constants is the f\/ield $\mathbf C(M)^{\mathcal F}$ of rational f\/irst integrals of the foliation.

Let $(M,\mathcal F)$ and $(M',\mathcal F')$ be foliated varieties. A regular (rational) map $\phi\colon (M',\mathcal F')\dasharrow (M,\mathcal F)$ is a regular (rational) morphism of foliated varieties if ${\rm d}\phi$ induces an isomorphism between $T_x\mathcal F'$ and $T_{\phi(x)}\mathcal F$ for (generic values of) $x\in M'$. It is clear that $\mathcal F'$ and $\mathcal F$ have the same rank.

A dif\/ferential f\/ield extension, correspond here to a dominant rational map of irreducible foliated varieties $\phi\colon (M',\mathcal F')\dasharrow (M,\mathcal F)$. It induces the extension $\phi^*\colon (\mathbf C(M),\mathfrak X_{\mathcal F})\to (\mathbf C(M'),\mathfrak X_{\mathcal F'})$ by composition with~$\phi$.

\begin{Example}Let $\mathcal F$ the foliation of $\mathbf C^2$ def\/ined by $\{{\rm d}y-y{\rm d}x = 0\}$. It corresponds to the dif\/ferential f\/ield $\big(\mathbf C(x,e^x), \big\langle\frac{{\rm d}}{{\rm d}x}\big\rangle\big) $.
\end{Example}

\begin{Remark}Throughout this appendix ``connection'' means ``f\/lat connection''.
\end{Remark}

\subsection[Invariant $\mathcal F$-connections]{Invariant $\boldsymbol{\mathcal F}$-connections}

Let us consider from now a foliated manifold of dimension $n$ and rank~$r$ without rational f\/irst integrals $(M,\mathcal F)$, an algebraic group $G$ and a~principal irreducible $G$-bundle $\pi\colon P\to M$. A~$G$-invariant connection in the direction of $\mathcal F$ is a foliation $\mathcal F'$ of rank $r$ in $P$ such that:
\begin{itemize}\itemsep=0pt
\item[(a)] $\pi\colon (P,\mathcal F')\to (M,\mathcal F)$ is a dominant regular map of foliated varieties;
\item[(b)] The foliation $\mathcal F'$ is invariant by the action of $G$ in $P$.
\end{itemize}
With this def\/inition $(\mathbf C(M), \mathfrak X_{\mathcal F})\to (\mathbf C(P), \mathfrak X_{\mathcal F'})$ is a dif\/ferential f\/ield extension. Also, each element $g\in G$ induces a dif\/ferential f\/ield automorphism of $(\mathbf C(P), \mathfrak X_{F'})$ that f\/ixes $(\mathbf C(M), \mathfrak X_F)$ by setting $(g\cdot f)(x) = f(x\cdot g)$.

Let $\mathfrak g$ be the Lie algebra of $G$. There is a way of def\/ining a $G$-equivariant form $\Theta_{\mathcal F'}$ with values in $\mathfrak g$, and def\/ined in ${\rm d}\pi^{-1}(T\mathcal F)$ in such way that $T\mathcal F'$ is the kernel of $\Theta_{\mathcal F'}$. First, there is a canonical form~$\Theta_0$ def\/ined in $\ker(d\pi)$ that sends each vertical vector $X_p\in \ker d_p\pi\subset T_pP$ to the element $\mathfrak g$ that verif\/ies,
\begin{gather*}\left.\frac{{\rm d}}{{\rm d}\varepsilon}\right|_{\varepsilon = 0} p \cdot \exp{\varepsilon A} = X_p.\end{gather*}
This form is $G$-equivariant in the sense that $R_g^*(\Theta_0) = \operatorname{Adj}_{g^{-1}} \circ \omega$. We have a decomposition of the vector bundle ${\rm d}\pi^{-1}(T\mathcal F) = \ker({\rm d}\pi) \oplus T\mathcal F'$. This decomposition allows to extend~$\Theta_0$ to a~form $\Theta_{\mathcal F'}$ def\/ined for vectors in ${\rm d}\pi^{-1}(T\mathcal F)$ whose kernel is precisely~$T\mathcal F$. We call \emph{horizontal frames} to those sections~$s$ of $\pi$ such that $s^*(\Theta_{\mathcal F}) = 0$.

\subsection{Picard--Vessiot bundle}

We say that the principal $G$-bundle with invariant $\mathcal F$-connection $\pi\colon (P,\mathcal F')\to (M,\mathcal F)$ is a~\emph{Picard--Vessiot} bundle if there are no rational f\/irst integrals of~$\mathcal F'$. The notion of Picard--Vessiot bundle corresponds exactly to that of primitive extension of Kolchin. In such case~$G$ is the group of dif\/ferential f\/ield automorphisms of $(\mathbf C(P), \mathfrak X_{F'})$ that f\/ix $(\mathbf C(M), \mathfrak X_F)$ and $(\mathbf C(M), \mathfrak X_F)\to (\mathbf C(P), \mathfrak X_{F'})$ is a strongly normal extension. Moreover, any strongly normal extension with constant f\/ield $\mathbf C$ can be constructed in this way (see \cite[Chapter~VI, Section~10, Theorem 9]{Kolchin}).

One of the most remarkable properties of strongly normal extensions is the Galois correspondence (from \cite[Chapter~VI, Section~4]{Kolchin}).

\begin{Theorem}[Galois correspondence] Assume that $(\mathbf C(M),\mathfrak X_{\mathcal F})\to (\mathbf C(P),\mathfrak X_{\mathcal F'})$ is strongly normal with group of automorphisms~$G$. Then, there is a bijection between the set of intermediate differential field extensions and algebraic subgroups of~$G$. To each intermediate differential field extension, it corresponds the group of automorphisms that fix such an extension point-wise. To each subgroup of automorphisms it corresponds its subfield of fixed elements.
\end{Theorem}

\subsection[The Picard--Vessiot bundle of an invariant $\mathcal F$-connection]{The Picard--Vessiot bundle of an invariant $\boldsymbol{\mathcal F}$-connection}

Let us consider an irreducible principal $G$-bundle $\pi\colon (P,\mathcal F') \to (M,\mathcal F)$ endowed with an invariant $\mathcal F$-connection $\mathcal F'$. We assume that $\mathcal F$ has no rational f\/irst integrals. A result of Bonnet (see \cite[Theo\-rem~1.1]{Bonnet}) ensures that for a very generic point in $M$ the leaf passing through such point is Zariski dense in $M$. Let us consider such a Zariski-dense leaf $\mathcal L$ of $\mathcal F$ in~$M$. Let us consider any leaf $\mathcal L'$ of $\mathcal F'$ in $P$ that projects by $\pi$ onto $\mathcal L$. Its Zariski closure is unique in the following sense:

\begin{Theorem}\label{uniqueness}Let $\mathcal L'$ and $\mathcal L''$ two leaves of $\mathcal F'$ whose projections by $\pi$ are Zariski dense in~$M$. Then, there exist an element $g\in G$ such that $\overline{\mathcal L'} \cdot g = \overline{\mathcal L''}$.
\end{Theorem}

\begin{proof}By construction, there is some $x\in\pi(\mathcal L')\cap\pi(\overline{\mathcal L''})$. Let us consider $p\in \pi^{-1}(\{x\})\cap \mathcal L'$ and $q\in\pi^{-1}(\{x\})\cap \overline{\mathcal L''}$. Since $p$ and $q$ are in the same f\/iber, there is a unique element $g\in G$ such that $p\cdot g = q$. By the $G$-invariance of the connection $\mathcal L'\cdot g$ is the leaf of $\mathcal F'$ that passes through $q$. The set
$\overline{\mathcal L''}$ is, by construction, union of leaves of $\mathcal F'$ and contains the point $q$. Thus, $\overline{\mathcal L' \cdot g} \subseteq \overline{\mathcal L''}$, and $\overline{\mathcal L'}\cdot g \subseteq \overline{\mathcal L''}$. Now, by exchanging the roles of $\mathcal L'$ and $\mathcal L''$, we prove that there is an element $h$ such that $\overline{\mathcal L''}\cdot h\subseteq \overline{\mathcal L'}$. It follows $h = g^{-1}$. This f\/inishes the proof.
\end{proof}

Let $L$ be the Zariski closure of $\mathcal L'$. Let us consider the algebraic subgroup
\begin{gather*}H = \{g\in G \colon L \cdot g = L\}\end{gather*}
stabilizing $L$. The projection $\pi$ restricted to $L$ is dominant, thus there is a Zariski open subset~$M^\star$ such that $\pi^\star \colon L^\star \to M^\star$ is surjective. Let us call $\mathcal F^\star$ the restriction of $\mathcal F'$ to $L^\star$. It follows that the bundle:
$\pi^\star \colon (L^\star,\mathcal F^\star) \to (M^\star, \mathcal F|_{M^\star})$ is a principal bundle of structure group $H$ called Picard--Vessiot bundle. The dif\/ferential f\/ield extension $(\mathbf C(M),\mathfrak X_{\mathcal F}) \to (\mathbf C(L^\star),\mathfrak X_{\mathcal F^\star})$ is the so-called Picard--Vessiot extension associated to the connection. The algebraic group $H$ is the dif\/ferential Galois group of the connection.

\subsection{Split of a connection}

Let us consider a pair of morphisms of foliated varieties
\begin{gather*}\phi_j\colon \ (M_j,\mathcal F_j)\to(M,\mathcal F),\qquad \mbox{for} \quad j=1,2.\end{gather*}
Then, we can def\/ine in $M_1\times_M M_2$ a foliation $\mathcal F_1\times_{\mathcal F} \mathcal F_2$ in the following way. A vector $X = (X_1,X_2)$ is in $T(\mathcal F_1\times_{\mathcal F}\mathcal F_2)$ if and only if ${\rm d}\phi_1(X_1)= {\rm d}\phi_2(X_2)\in T\mathcal F$. Let us consider $(P,\mathcal F')$ a~principal $\mathcal F$ connection. Note that the projection
\begin{gather*}\pi_1 \colon \ (M_1\times_M P, \mathcal F_1\times_{\mathcal F} \mathcal F')\to (M_1,\mathcal F_1)\end{gather*}
is a~principal $G$-bundle endowed of a $\mathcal F_1$-connection. We call this bundle the pullback of $(P,\mathcal F')$ by $\phi_1$.

We also may consider the trivial $G$-invariant connection $\mathcal F_0$ in the trivial principal $G$-bundle
\begin{gather*}\pi_0\colon \ (M\times G,\mathcal F_0) \to (M,\mathcal F),\end{gather*}
for what the leaves of $\mathcal F_0$ are of the form $(\mathcal L, g)$ where $\mathcal L$ is a leaf of~$\mathcal F$ and~$g$ a f\/ixed element of~$G$. We say that the $G$-invariant connection~$(P,\mathcal F')$ is rationally trivial if there is a birational $G$-equivariant isomorphism of foliated manifolds between $(P,\mathcal F)$ and $(M\times G, \mathcal F_0)$.

Invariant connections are always trivialized after pullback; there is a universal $G$-equivariant isomorphism def\/ined over~$P$
\begin{gather*}(P \times G, \mathcal F'\times_{\mathcal F} \mathcal F_0) \to (P\times_M P, \mathcal F'\times_{\mathcal F}\mathcal F'), \qquad (p, g) \mapsto (p, p\cdot g),\end{gather*}
that trivializes any $G$-invariant connection. However, the dif\/ferential f\/ield $(\mathbf C(P),\mathfrak X_{\mathcal F'})$ may have new constant elements. To avoid this, we replace the pullback to $P$ by a pullback to the Picard--Vessiot bundle~$L^\star$
\begin{gather*}(L^\star \times G, \mathcal F^\star\times_{\mathcal F} \mathcal F_0) \to (L^\star\times_M P, \mathcal F^\star\times_{\mathcal F}\mathcal F'), \qquad (p, g) \mapsto (p, p\cdot g).\end{gather*}

The Picard--Vessiot bundle has some minimality property. It is the smallest bundle on $M$ that trivializes the connection. We have the following result.

\begin{Theorem}\label{uniqueness2} Let us consider $\pi\colon (P,\mathcal F') \to (M,\mathcal F)$ be as above, $\pi^\star \colon (L^\star,\mathcal F^\star)\to (M,\mathcal F)$ the Picard--Vessiot bundle, and and $\phi\colon \big(\tilde M, \tilde{\mathcal F}\big) \to (M, \mathcal F)$ any dominant rational map of foliated varieties such that:
\begin{itemize}\itemsep=0pt
\item[$(a)$] $\tilde{\mathcal F}$ has no rational first integrals in $\tilde M$;
\item[$(b)$] the pullback $\big(\tilde M \times_M P,\tilde F \times_{\mathcal F}\mathcal F'\big)\to \big(\tilde M, \tilde{\mathcal F}\big)$ is rationally trivial.
\end{itemize}
There is a dominant rational map of foliated varieties $\psi\colon \tilde M \dasharrow L^\star$ such that $\pi^\star\circ\psi = \phi$ in their common domain.
\end{Theorem}

\begin{proof} Let us take $\tau \colon \tilde M \times G \dasharrow \tilde M \times_M P$ a birational trivialization, $\pi_2 \colon \tilde M \times_M P \to P$ be the projection in the second factor, and $\iota \colon \tilde M \to \tilde M \times G$ the inclusion $p \mapsto (p,e)$. Then, $\tilde\psi = \pi_2 \circ \tau \circ \iota$ is a rational map from $\tilde M$ to $P$ whose dif\/ferential sends $T\tilde{\mathcal F}$ to $T\mathcal F$. By Bonnet theorem, $\tilde M$~is the Zariski closure of a leaf of $\tilde{\mathcal F}$ that projects by $\phi$ into a Zariski dense leaf of~$\mathcal F$. From this, $\tilde\psi$~contains a dense leaf of $\mathcal F'$ in $P$. By applying a suitable right translation in $P$ and the uniqueness Theorem~\ref{uniqueness}, we obtain the desired conclusion.
\end{proof}

\subsection{Linear connections} \label{A6}

Let $(M,\mathcal F)$ be as above, of dimension $n$ and rank $r$. Let $\xi\colon E\to M$ be a vector bundle of rank~$k$. A linear integrable $\mathcal F$-connection is a foliation $\mathcal F_E$ of rank~$r$ which is compatible with the structure of vector bundle in the following sense: the point-wise addition of two leaves of any dilation of a leaf is also a leaf. This can also be stated in terms of a covariant derivative operator~$\nabla$ wich is def\/ined only in the direction of~$\mathcal F$. First, the kernel of ${\rm d}\xi$ is naturally projected onto~$E$ itself
\begin{gather*}{\rm vert}_0 \colon \ \ker({\rm d}\xi) \to E, \qquad X_v \mapsto w,\end{gather*}
where $\left.\frac{{\rm d}}{{\rm d}\varepsilon}\right|_{\varepsilon = 0} v + \varepsilon w = X_v$. Then, the decomposition of ${\rm d}\xi^{-1}(T\mathcal F)$ as $\ker({\rm d}\xi)\oplus T\mathcal F_E$ allows us to extend ${\rm vert_0}$ to a projection
\begin{gather*}{\rm vert}\colon \ {\rm d}\xi^{-1}(T\mathcal F) \to E.\end{gather*}
Thus, we def\/ine for each section $s$ its covariant derivative $\nabla s = s^*({\rm vert}\circ {\rm d}s|_{T\mathcal F})$. This is a $1$-form on~$M$ def\/ined only for vectors in $T \mathcal F$. This covariant derivative has the desired properties, it is additive and satisf\/ies the Leibniz formula
\begin{gather*}\nabla (fs) = {\rm d}f|_{T\mathcal F}\otimes s + f \nabla s.\end{gather*}
In general, we write for $X$ a vector in $T\mathcal F$, $\nabla_X s$ for the contraction of $\nabla s$ with the vector $X$. It is an element of $E$ over the same base point in $M$ that the vector $X$. We call \emph{horizontal sections} to those sections $s$ of $\xi$ such that $\nabla s = 0$.

Let $\pi\colon R^1(E)\to M$ be the bundle of linear frames in $E$. It is a principal linear ${\rm GL}_k(\mathbf C)$-bundle. The foliation $\mathcal F_E$ induces a foliation $\mathcal F'$ in $R^1(E)$ that is a $G$-invariant $\mathcal F$-connection. Let us consider the Picard--Vessiot bundle, $(L^\star,\mathcal F^\star)$. The uniqueness Theorem~\ref{uniqueness2} on the Picard--Vessiot bundle, can be rephrased algebraically in the following way. The Picard--Vessiot extension $(\mathbf C(M), \mathfrak X_{\mathcal F}) \to (\mathbf C(L^\star), \mathfrak X_{\mathcal F^\star})$ is characterized by the following properties (cf.~\cite[Section~1.3]{SingerVanderput}):
\begin{itemize}\itemsep=0pt
\item[(a)] there are no new constants, $\mathbf C(L^\star) = \mathbf C$;
\item[(b)] it is spanned, as a f\/ield extension of $\mathbf C(M)$, by the coef\/f\/icients of a fundamental matrix of solutions of the dif\/ferential equation of the horizontal sections.
\end{itemize}

\subsection{Associated connections}\label{ap7_associated}

Let $\pi\colon (P,\mathcal F')\to (M,\mathcal F)$ be a $G$-invariant connection, as before, where $\mathcal F$ is a foliation in $M$ without rational f\/irst integrals. Let us consider $\phi\colon G\to {\rm GL}(V)$ a f\/inite-dimensional linear representation of~$G$. It is well known that the associated bundle $\pi_P \colon V_P \to M$,
\begin{gather*}V_P = P\times_G V = (P\times V)/G, \qquad (p\cdot g,v) \sim (p,g\cdot v),\end{gather*}
is a vector bundle with f\/iber $V$. Here we represent the action of $G$ in $V$ by the same operation symbol than before. The $G$-invariant connection
$\mathcal F'$ rises to a foliation in $P\times G$ and then it is projected to a foliation $\mathcal F_V$ in $V_P$. In this way, the projection
\begin{gather*}\pi_P\colon \ (V_P,\mathcal F_V)\to (M,\mathcal F),\end{gather*}
turns out to be a linear $\mathcal F$-connection. It is called the \emph{Lie--Vessiot} connection induced in the associated bundle. The Galois group of the principal and the associated Lie--Vessiot connection are linked in the following way.

\begin{Theorem}\label{associated galois}Let $H\subset G$ be the Galois group of the principal connection $\mathcal F'$. Then, the Galois group of the associated Lie--Vessiot connection $\mathcal F_V$ is $\phi(H)\subseteq {\rm GL}(V)$.
\end{Theorem}

\begin{proof}Let us consider the bundle of frames $R^1(V_P)$, with its induced invariant connection $\mathcal F''$. Let us f\/ix a basis $\{v_1,\ldots,v_r\}$ of $V$. Then, we have a map
\begin{gather*}\tilde \pi \colon \ P \to R^{1}(V_P), \qquad p\mapsto ([p, v_1],\ldots, [p, v_r]),\end{gather*}
where the pair $[p, v]$ represents the class of the pair $(p,v)\in P\times V$. By construction, $\tilde\pi$ sends $T\mathcal F'$ to $T\mathcal F''$. It implies that, if $\mathcal L^{\star}$ is a Picard--Vessiot bundle for $\mathcal F'$ then $\tilde\pi(L^\star)$ is a Picard--Vessiot bundle for $\mathcal F''$. Second, if $\mathcal L^\star$ is a principal $H$ bundle, then $\tilde\pi(L^\star)$ is a principal $H/K$ bundle where $K$ is the subgroup of $H$ that stabilizes the basis $\{v_1,\ldots,v_r\}$.
\end{proof}

Let us discuss how the covariant derivative operator in $\nabla$ is def\/ined in terms of $\Theta_{\mathcal F'}$ and the action of $G$ in $V$. Let us denote by $\phi'\colon \mathfrak g\to \mathfrak{gl}(V)$ the induced Lie algebra morphism. Let~$s$ be a local section of~$\xi$. Let us consider the canonical projection $\bar\pi \colon P \times V \to V(P)$. This turns out to be also a principal bundle, here the action on pairs is $(p,v)\cdot g = \big(p \cdot g, g^{-1} \cdot v\big)$. Now we can take any section $r$ of this bundle, and def\/ine $\tilde s = r\circ s$. As $r$ takes values in a~cartesian product, we obtain $\tilde s = (s_1, s_2)$ where $s_1$ is a section of $\pi$ and $s_2$ is a function with values in~$V$. Finally we obtain
\begin{gather}\label{covariant_associated}
\nabla s = {\rm d}s_2|_{T\mathcal F} - \phi'(s_1^*(\Theta_{\mathcal F'}))(s_2).
\end{gather}
A calculation shows that it does not depend of the choice of $r$ and it is the covariant derivative operator associated to~$\mathcal F_V$. In particular, if $s_2$ is already an horizontal frame, then the covariant dif\/ferential is given by the f\/irst term ${\rm d}s_s|_{T\mathcal F}$.

\section{Deligne's realization of Lie algebra}\label{apB}

The proof of the existence of a regular parallelism for any complex Lie algebra $\mathfrak g$ is written in a~set of two letters from P.~Deligne to B.~Malgrange (dated from November of~2005 and February of~2010 respectively) that are published verbatim in~\cite{Malgrange}. We reproduce here the proof with some extra details.

\begin{Theorem}[Deligne]\label{TDeligne}Given any complex Lie algebra $\mathfrak g$ there exist an algebraic variety endowed with a regular parallelism of type~$\mathfrak g$.
\end{Theorem}

\begin{Lemma}\label{ap2_1}Let $T$ be an algebraic torus acting regularly by automophisms in some algebraic group $H$ and let $\mathfrak t$ be the Lie algebra of $T$. Let us consider the semidirect product
\begin{gather*}\mathfrak t \ltimes H, \qquad (t,h)(t',h') = (t+t', (\exp(t')\cdot h)h')\end{gather*}
as an algebraic variety and analytic Lie group. Its left invariant vector fields form a regular parallelism of $\mathfrak t \ltimes H$. The Galois group of this parallelism is a~torus.
\end{Lemma}

\begin{proof} Let us denote by $\alpha$ the action of $T$ in $H$ and $\alpha'\colon \mathfrak t \mapsto \mathfrak X[H]$ the Lie algebra isomorphism given by the inf\/initesimal generators
\begin{gather*}(\alpha'X)_h= \left.\frac{{\rm d}}{{\rm d}\varepsilon}\right|_{\varepsilon=0} \alpha_{\exp(\varepsilon t)}(h).\end{gather*}
Let $X$ be an invariant vector f\/ield in~$\mathfrak t$. Let us compute the left invariant vector f\/ield in $t\ltimes H$ whose value
at the identity is $(X_0,0)$. In order to perform the computation we write the vector as an inf\/initesimally near point to $(0,e)$.
\begin{gather*}L_{(t,h)}(0 + \varepsilon X_0, e) = (t + \varepsilon X_t, \alpha_{\exp(\varepsilon X)}(h) ) = (t + \varepsilon X_t, h + \varepsilon (\alpha'X)_h).\end{gather*}
And thus $dL_{(t,h)}(X_0,0) = (X_t, (\alpha'X)_h)$. We conclude that $(X,\alpha'X)\in\mathfrak X[\mathfrak t\ltimes H]$ is the left invariant vector f\/ield whose value at $(0,e)$ is $(X_0,0)$. Let us consider now $Y$ a left invariant vector f\/ield in $H$. Let us compute, as before, the left invariant vector f\/ield whose value at $(t,h)$ is~$(0,Y_h)$
\begin{gather*}L_{(t,h)}(0,e+\varepsilon Y_e) = (t, L_h(e + \varepsilon Y_e)) = (t, h + \varepsilon Y_h).\end{gather*}
And thus $(0,Y)$ is the left invariant vector f\/ield whose value at $(0,e)$ is $(0,Y_e)$. These vector f\/ields of the form $(X,\alpha'X)$ and~$(0,Y)$ are regular and span the Lie algebra of left invariant vector f\/ields in $\mathfrak t\ltimes H$. Hence, they form a regular parallelism.

In order to compute the Galois group of the parallelism, let us compute its reciprocal parallelism. It is formed by the right invariant vector f\/ields in the analytic Lie group $\mathfrak t \ltimes H$. A similar computation proves that if $X$ is an invariant vector f\/ield in~$\mathfrak t$ then $(X,0)$ is right invariant in $\mathfrak t \ltimes H$. For each element $\tau\in T$, $\alpha_\tau$ is a group automorphism of~$H$. Thus, it induces a derived automorphism $\alpha_{\tau*}$ of the Lie algebra of regular vector f\/ields in~$H$. Let $Y$ be now a right invariant vector f\/ield in $H$. Let us compute the right invariant vector f\/ield $Z$ in $\mathfrak t\ltimes H$ whose value at~$(0,e)$ is $(0,Y_e)$:
\begin{gather*}R_{(t,h)}(0,e+ \varepsilon Y_e) = (t,\alpha_{\exp(t)}(e+ \varepsilon Y_e)h) = (t, h + \varepsilon (\alpha_{\exp(t)*}Y)_h)\end{gather*}
and $Z_{t,h} = (0, (\alpha_{\exp(t)*} Y)_h)$. Those analytic vector f\/ields depend on the exponential function in a torus thus we can conclude,
by a standard argument of dif\/ferential Galois theory, that the associated dif\/ferential Galois group is a torus.
\end{proof}

Let us consider $\mathfrak g$ an arbitrary, non algebraic, f\/inite-dimensional complex Lie algebra. We consider an embedding of $\mathfrak g$ in the Lie algebra of general linear group and $E$ the smallest algebraic subgroup whose Lie algebra $\mathfrak e$ contains~$\mathfrak g$. $E$ is a~connected linear algebraic group.

\begin{Lemma}[{also in \cite[Proposition~1]{Chevalley}}]\label{ap2_2}
With the above definitions and notation $[\mathfrak e, \mathfrak e] = [\mathfrak g, \mathfrak g]$.
\end{Lemma}

\begin{proof}Let $H$ be the group of matrices that stabilizes $\mathfrak g$ and acts trivially on $\mathfrak g/[\mathfrak g, \mathfrak g]$. Its Lie algebra $\mathfrak h$ contains $\mathfrak g$ and thus $H\supseteq E$ and $\mathfrak h \supseteq \mathfrak e$. By def\/inition of $H$ we have $[\mathfrak h, \mathfrak g] = [\mathfrak g, \mathfrak g]$, therefore $[\mathfrak e, \mathfrak g]\subseteq [\mathfrak g,\mathfrak g]$. Let us now consider the group $H_1$ that stabilizes $\mathfrak e$ and $\mathfrak g$ and that acts trivially in $\mathfrak e/[\mathfrak g, \mathfrak g]$. This is again an algebraic group containing~$E$, and its Lie algebra $\mathfrak h_1$ satisf\/ies $[\mathfrak h_1,\mathfrak e] \subseteq [\mathfrak g, \mathfrak g]$. Taking into account $\mathfrak e \subseteq \mathfrak h_1$ we have $[\mathfrak e, \mathfrak e] \subseteq [\mathfrak g, \mathfrak g]$. The other inclusion is trivial.
\end{proof}

Because of Lemma \ref{ap2_2}, the abelianized Lie algebra $\mathfrak g^{ab} = \mathfrak g / [\mathfrak g, \mathfrak g]$ is a~subspace of $\mathfrak e^{ab} = \mathfrak e/ [\mathfrak e, \mathfrak e]$. Moreover, if we consider the quotient map, $\pi\colon \mathfrak e \to \mathfrak e^{ab}$, then $\mathfrak g = \pi^{-1}(\mathfrak g^{ab})$.

Let us consider an algebraic Levy decomposition $E\simeq L \ltimes U$ (see \cite[Chapter~6]{Onishchik}). Here, $L$ is reductive and $U$ is the unipotent radical, consisting in all the unipotent elements of~$E$. The semidirect product structure is produced by an action of~$L$ in~$U$, so that, $(l_1,u_1)(l_2,u_2) = (l_1l_2, (l_2\cdot u_1)u_2)$.

Since $L$ is reductive, its commutator subgroup $L'$ is semisimple. Let $T$ be the center of $L$, which is a torus, the map
\begin{gather*}\varphi\colon \ T\times L' \to L, \qquad (t,l) \mapsto tl,\end{gather*}
is an isogeny. The isogeny def\/ines an action of $T\times L'$ in $U$ by $(t,l)\cdot u = tl\cdot u$. We have found an isogeny
\begin{gather*}(T\times L')\ltimes U \to E.\end{gather*}

The Lie algebra $\mathfrak u$ of $U$ is a nilpotent Lie algebra, so that the exponential map $\exp\colon \mathfrak u \to U$ is regular and bijective. In general, if $V$ is an abelian quotient of $U$ with Lie algebra $\mathfrak v$ then the exponential map conjugates the addition law in $\mathfrak v$ with the group law in~$V$.

\begin{Lemma}\label{ap2_3} With the above definitions and notation, let $\bar{\mathfrak u}$ be the biggest quotient of $\mathfrak u^{ab}$ in which $L$ acts by the identity. We have a Lie algebra isomorphism $\mathfrak e^{ab} \simeq \mathfrak t \times \bar{\mathfrak u}$.
\end{Lemma}

\begin{proof} Let us compute $\mathfrak e^{ab}$. We compute the commutators $\mathfrak e$ by means of the isomorphism $\mathfrak e \simeq (\mathfrak t \times \mathfrak l') \ltimes \mathfrak u$. We obtain
\begin{gather*}[(t_1,l_1,u_1),(t_2,l_2,u_2)] = (0,[l_1,l_2], a(t_2,l_2)u_1 +[u_1,u_2]),\end{gather*}
where $a$ represents the derivative at $(e,e)$ of the action of $L$ in $U$. From this we obtain that $[\mathfrak e, \mathfrak e]$
is spanned by $(\{0\}\times \mathfrak l')\ltimes \mathfrak \{0\}$, $\{0\}\ltimes [\mathfrak u, \mathfrak u]$ and $\{0\}\times \langle a(\mathfrak l)\mathfrak u \rangle$. Taking into account that $\bar{\mathfrak u}/ ( \langle a(\mathfrak l)\mathfrak u \rangle + [\mathfrak u, \mathfrak u] )$ is the biggest quotient of $\mathfrak u^{ab}$ in which $L$ acts trivially, we obtain the result of the lemma.
\end{proof}

Let $\mathfrak t$ be the Lie algebra of $T$. Its exponential map is an analytic group morphism and thus we may consider the analytic action of $\mathfrak t \times L'$ in $U$ given by $(t,l)\cdot u = (\exp(t) l )\cdot u$. Let $\tilde E$ be the algebraic variety and analytic Lie group $(\mathfrak t \times L') \ltimes U$. By application of Lemma~\ref{ap2_1}, and taking into account that $\tilde E \simeq \mathfrak t \ltimes H$, where $H$ is the group $L'\cdot U$, we have that the left invariant vector f\/ields in $\tilde E$ are regular. Let us consider the projection
\begin{gather*}\pi_1\colon \ \tilde E \to \mathfrak e^{ab} = \mathfrak t \times \bar{\mathfrak u}, \qquad (t,l,u) \mapsto (t,[\log(u)]),\end{gather*}
this projection is algebraic by construction, and also a morphism of Lie groups. By Lemmas~\ref{ap2_2} and~\ref{ap2_3}, $\mathfrak g^{ab}$ is a vector subspace of the image. Then, let us take $\tilde G$ the f\/iber $\pi^{-1}_1(\mathfrak g^{ab})$. It is an algebraic submanifold of $\tilde E$ and an analytic Lie group. The derivative at the identity of $\pi_1$ is precisely the abelianization $\pi$ and it follows that the Lie algebra of $\tilde G$ is precisely~$\mathfrak g$. Finally~$\tilde G$ is an algebraic variety with a~regular $\mathfrak g$-parallelism. This f\/inishes the proof of Theorem~\ref{TDeligne}.

\begin{Remark}The right invariant vector f\/ields in~$\tilde G$ are constructed as in Lemma~\ref{ap2_1} by means of the exponential function in the torus. Hence, Galois groups of the parallelisms obtained via this construction are always tori.
\end{Remark}

\subsection*{Acknowledgements}

The authors thank the ECOS-Nord program C12M01 and the project ``IsoGalois'' ANR-13-JS01-0002-01. They also thank the ``Universidad Nacional de Colombia''(project HERMES code 37243) and the ``Universit\'e de Rennes 1'' (Actions Internationales 2016) for supporting this reseach, and also the Centre Henri Lebesgue ANR-11-LABX-0020-01 for creating an attractive mathematical environment.

The authors thank Juan Diego V\'elez for his help with the f\/inal redaction of the manuscript and the anonymous referees who gave relevant contributions to improve the paper.

\pdfbookmark[1]{References}{ref}
\LastPageEnding

\end{document}